\documentclass[a4paper,11pt,english]{amsart}

\usepackage{multirow}
\usepackage[english]{babel}
\usepackage[latin1]{inputenc}
\usepackage[T1]{fontenc}
\usepackage{wasysym}

\usepackage{graphicx}
\usepackage{hyperref}
\usepackage{amssymb}
\usepackage{amsmath}
\usepackage{verbatim,enumerate}
\usepackage{tikz}

\newtheorem{thm}{Theorem}[section]
\newtheorem{lemma}[thm]{Lemma}
\newtheorem{prop}[thm]{Proposition}
\newtheorem{cor}[thm]{Corollary}
\newtheorem{defi}[thm]{Definition}
{\theoremstyle{definition}
\newtheorem{exa}[thm]{Example}
\newtheorem{rem}[thm]{Remark}}
\newcommand{\C}{\mathbb{C}}
\newcommand{\CC}{\mathcal{C}}
\newcommand{\CP}{\mathbb{C}P}
\newcommand{\RP}{\mathbb{R}P}

\renewcommand{\epsilon}{\varepsilon}
\newcommand{\R}{\mathbb{R}}
\newcommand{\Z}{\mathbb{Z}}

\newcommand{\x}{\underline{x}}

\renewcommand{\qedsymbol}{\smiley}

\begin{document}

\title{Surgery of real symplectic  fourfolds and
Welschinger invariants}

\author{Erwan Brugall\'e}
\address{Erwan Brugall\'e, CMLS, École polytechnique, CNRS, Université
  Paris-Saclay, 91128 Palaiseau Cedex, France; Université de Nantes, Laboratoire de
  Mathématiques Jean Leray, 2 rue de la Houssinière, F-44322 Nantes Cedex 3,
France}
\email{erwan.brugalle@math.cnrs.fr}


\begin{abstract}A  surgery of a real symplectic manifold $X_\R$ along a real Lagrangian sphere
$S$ is a modification of the symplectic and real structure on $X_\R$ in a
neighborhood of $S$.  Genus 0 Welschinger invariants of two real symplectic
$4$-manifolds differing by such a  surgery have been related in
\cite{BP14}. In the present paper, we
explore some particular
situations where general formulas from \cite{BP14} greatly simplify. As an
application, we reduce the computation of genus 0 
Welschinger invariants of all del~Pezzo surfaces 
 to the cases covered by \cite{Bru14}, and  of all $\R$-minimal real conic bundles 
to the cases covered by \cite{HorSol12}. As a by-product,
we establish the existence of some new relative Welschinger 
invariants.
We also generalize results from \cite{BP14} to the enumeration of
curves  of higher
genus, and give relations between hypothetical invariants defined in
the same vein as \cite{Shu14}. 
\end{abstract}
\maketitle
\tableofcontents

Let $\Lambda$  be either $\Z$,  $\Z/2\Z$, or
$\mathbb Q$.
Given a (oriented if $\Lambda=\Z$ or $\mathbb Q$) smooth compact manifold  $X$ of
dimension 4, 
the intersection product of two elements
$d_1,d_2\in H_2(X;\Lambda)$ is
denoted by $d_1\cdot d_2\in \Lambda$.
The class realized in  $H_2(X;\Lambda)$ by 
 a $2$-cycle $C$ in  $X$ is denoted by
$[C]$. The subgroup orthogonal to a class $\delta\in H_2(X;\Lambda)$ for the
intersection form is denoted by $\delta^\perp$.

A \textit{real symplectic manifold} $X_\R=(X,\omega_X,\tau_X)$ is a symplectic
manifold $(X,\omega_X)$ equipped with an anti-symplectic involution
$\tau_X$. The \textit{real part} of $(X,\omega_X,\tau_X)$, denoted by
$\R X$, is by
definition the fixed point set of $\tau_X$.
A projective real algebraic variety  is always implicitly assumed to be
equipped with some
Kähler form which turns it into a real symplectic manifold.
Two symplectic forms $\omega_X$ and $\omega_X'$
(resp. two real symplectic structures $(\omega_X,\tau_X)$ and $(\omega_X',\tau_X')$) on a
manifold $X$ 
are said to be \emph{deformation equivalent} if there exists a
smooth family of symplectic forms connecting $\omega_X$ and
$\omega_X'$ (resp. a smooth family of real symplectic structures
connecting  $(\omega_X,\tau_X)$ to $(\omega_X',\tau_X')$). Two
symplectic manifolds $(X,\omega_X)$ and $(Y,\omega_Y)$ 
(resp. two real symplectic manifolds $(X,\omega_X,\tau_X)$ and $(Y,\omega_Y,\tau_Y)$)
are said to be \emph{deformation equivalent} if one can pass from one
to the other by a finite sequence of deformations of symplectic form and
symplectomorphisms (resp.  deformations of real symplectic structure and
equivariant symplectomorphisms).

In this text, the manifold  $X$ will always be $4$-dimensional, and we denote by
$H_2^{\tau_X}(X;\Lambda)$ the space of $\tau_X$-invariant classes, and  by
$H_2^{-\tau_X}(X;\Lambda)$ the space of $\tau_X$-anti-invariant classes.

\section{Introduction}
Beside
blow-up, surgery along  a real Lagrangian sphere
 is a natural  and elementary operation
on  \emph{real} algebraic or symplectic
manifolds.
For example, there exists only  three real rational algebraic
surfaces 
up to deformation, blow-up, and surgery along a real Lagrangian
sphere: $\CP^2$ equipped with its standard real structure, and
$\CP^1\times\CP^1$ equipped with any of its two non-equivalent real
structures with an empty real part\footnote{The classification up to
  deformation and blow-up is 
given in \cite{DegKha02}; the classification up to surgery along a real
Lagrangian sphere follows then from the rigid isotopy classifications
of plane real quartics, of real cubic sections of the quadratic cone
in $\CP^3$, and of real quadrics in $\CP^3$, see for example \cite{DK}.}.
Welschinger invariants are invariant under deformation, hence
understanding how  they
behave under blow-up and
surgery along a real Lagrangian sphere would allow for reducing
significantly the basic ambient real symplectic manifolds in which to
perform actual computations.

  \medskip
In this paper we relate, under mild assumptions,  Welschinger invariants
of two real symplectic  4-manifolds differing by 
such  surgery. 
We start by describing informally this operation,
 and we refer to
Section \ref{sec:real surg} for  precise definitions. 
Let $X_\R=(X,\omega_X,\tau_X)$ be a real compact symplectic manifold of
dimension 4, 
and let $S\subset X$ be a real Lagrangian sphere, 
i.e. a Lagrangian sphere globally invariant under $\tau_X$.
It follows from Weinstein Lagrangian neighborhood Theorem  (see for
example \cite[Theorem 3.33]{Duff})
that  there exists a real open symplectic embedding
of a neighborhood $V$ of $S$ 
to
the real affine quadric $(Q,\omega_Q, \tau)$ in $\C^3$ given
by the equation
$$(-1)^{\epsilon_1} x^2+(-1)^{\epsilon_2} y^2+(-1)^{\epsilon_3} z^2=1
\qquad \mbox{with }\epsilon_i\in\{0,1\},$$
and $S$ to  the sphere $S_Q$ in
$i^{\epsilon_1}\R\times i^{\epsilon_2}\R \times i^{\epsilon_3}\R $
with equation
$$x^2+y^2+z^2=1. $$
Observe that the automorphism $(x,y,z)\mapsto (-x,-y,-z)$ of $\C^3$ provides
another real structure $\tau'$ on  $(Q,\omega_Q)$ that coincide
with $\tau$ at infinity, and for which $S_Q$ remains globally invariant.
As a consequence,
one can modify the symplectic and real structure of  $X_\R$ in $V$ so that $V$ now
becomes equivariantly symplectomorphic to a bounded open subset of the
real affine quadric  in $\C^3$ with equation
$$(-1)^{\epsilon_1} x^2+(-1)^{\epsilon_2} y^2+(-1)^{\epsilon_3} z^2=-1 .$$  
The resulting real symplectic manifold 
$Y_\R$ is called a \emph{surgery of $X_\R$ along $S$}. From
this local description,
we see that
(with the convention that $\chi(\emptyset)=0$)
$$\chi(\R Y)=\chi(\R X)\pm 2,$$
and that the class $[S]$ in $H_2(X;\Z)$ is $\tau_X$-anti-invariant if and
only if it is $\tau_Y$-invariant (in which case we have 
$\chi(\R Y)=\chi(\R X)+ 2$).
Note that  $X_\R$ and $Y_\R$ have the same
underlying smooth
manifold, and that  $Y_\R$ only depends, up to deformation,  on $X_\R$
and on the class
realized by $S$ in $H_2(X;\Z)$. Since we are interested in this paper
in invariants under deformation of real symplectic manifolds, we say
that $Y_\R$ is \emph{the} surgery of $X_\R$ along $S$ rather than
\emph{a} surgery.

The two above real quadrics can be put into the real family $Q_t$ of 
quadrics with equation
$$(-1)^{\epsilon_1} x^2+(-1)^{\epsilon_2} y^2+(-1)^{\epsilon_3} z^2=t   
 \qquad \mbox{with }|t|\le 1.$$  
The quadric $Q_0$ is the unique singular quadric of the family. Hence
$X_\R$ and $Y_\R$
can be 
represented as 
two different real fibers of a real Lefschetz fibration of over a disk, having
a unique singular fiber
for which $S$ realizes precisely the vanishing cycle
(see Figure \ref{fig:degen}).
\begin{figure}[h!]
\begin{center}
\begin{tabular}{ccc}
\includegraphics[width=2cm, angle=0]{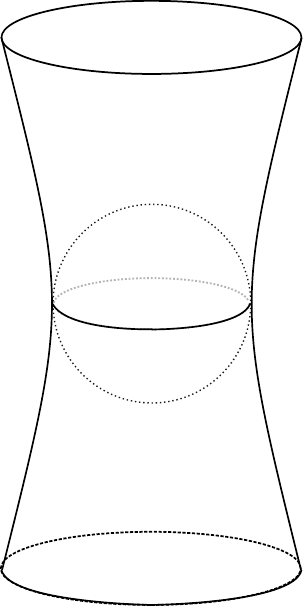}
\hspace{3ex}
\includegraphics[width=2cm, angle=0]{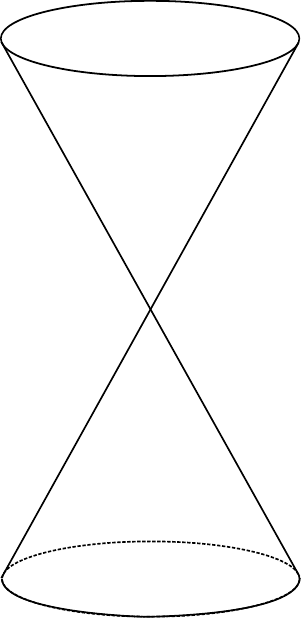}
\hspace{3ex} 
\includegraphics[width=2cm, angle=0]{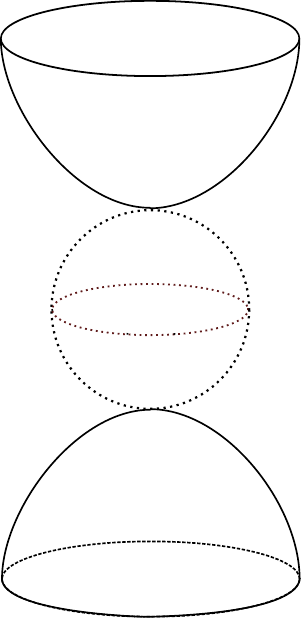}
\put(-195,-15){$\R Q_{1}$}
\put(-195,85){$S_{Q_1}$}
\put(-115,-15){$\R Q_{0}$}
\put(-35,-15){$\R Q_{-1}$}
\put(-5,55){$S_{Q_{-1}}$}
& \hspace{5ex}
&\includegraphics[width=2cm, angle=0]{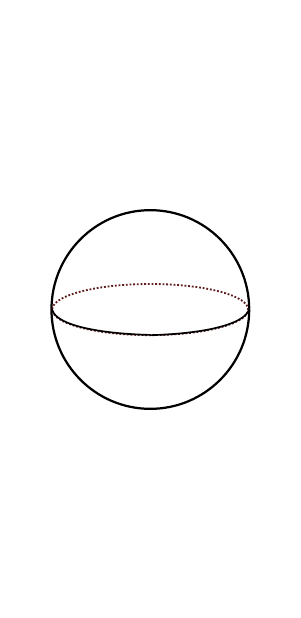}
\hspace{1ex} 
\includegraphics[width=2cm, angle=0]{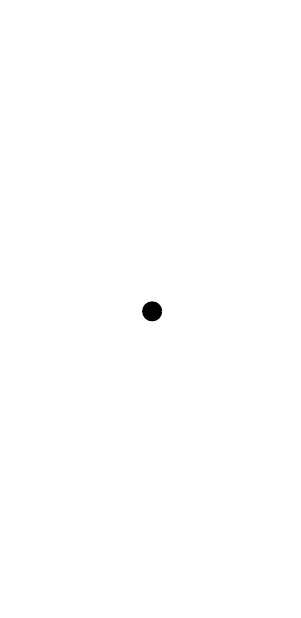}
\hspace{1ex} 
\includegraphics[width=2cm, angle=0]{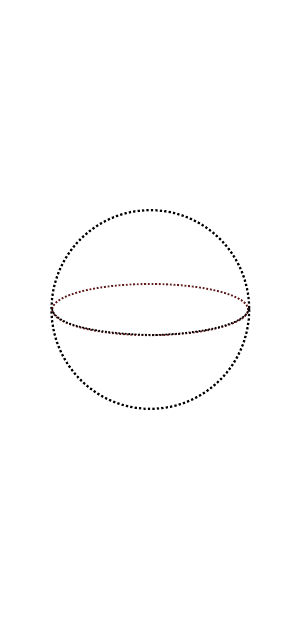}
\put(-185,-15){$\R Q_{1}=S_{Q_1}$}
\put(-100,-15){$\R Q_{0}$}
\put(-50,-15){$\R Q_{-1}=\emptyset$}
\put(-5,55){$S_{Q_{-1}}$}
\\ \\a) $Q_t$ with equation $x^2+ y^2- z^2=t$
&&b) $Q_t$ with equation $x^2+ y^2+ z^2=t$
\end{tabular}
\end{center}
\caption{$Q_t$ with equation $x^2+ y^2\pm z^2=t$.}
\label{fig:degen}
\end{figure}
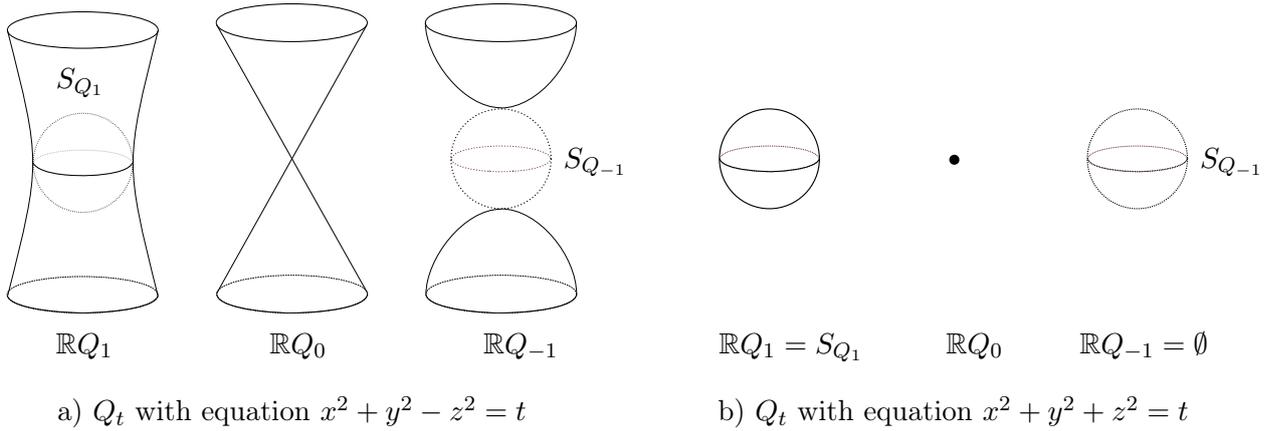

\medskip
Now we specialize our main statement, Theorem \ref{thm:main1}, for a
particular  type of Welschinger invariants that are easy to
define. Choose the following:
\begin{itemize}
\item a connected component $L$ of $\R X$,
\item $R$ which is  either the empty set or the set
  $\R X\setminus L$;  denote by $F$ the class realized by $R$ in 
 $H_2^{\tau_X}(X\setminus L;\Z/2\Z)$;
\item  a class $d\in H_2^{-\tau_X}(X;\Z)$,
\item   $r,s\in\Z_{\ge 0}$ such that  $c_1(X)\cdot d  - 1= r+2s,$
  \item  a configuration $\x$ made of $r$ points in $L$ and  $s$ 
 pairs of $\tau_X$-conjugated 
 points in $X\setminus \R X$.
\end{itemize}
Given  an  almost complex structure $J$ tamed by $\omega_X$ for
which $\tau_X$ is $J$-antiholomorphic (i.e.
$J\circ d \tau_X=-d\tau_X\circ J$), 
we denote by $\mathcal C(d,\x,L,J)$
 the set of  real rational 
$J$-holomorphic curves $f:\CP^1 \to X$, with $f_*[\CP^1]=d$, passing through $\x$,
 and such that $f(\RP^1)\subset L$. For a generic choice of $J$, 
 the set $\mathcal C(d,\x,L,J)$ is finite
and composed of immersions. Given an element $f:\CP^1\to X$ of $\mathcal
C(d,\x,L,J)$, 
 we
define the
 $(L,F)$-mass $m_{L,F}(f)$ of $f$ as
 the number of  elliptic real  nodes of  $f(\CP^1)$ (i.e.
real nodes with two
$\tau_X$-conjugated branches) contained in $L\cup R$.
The number
$$W_{X_\R,L,F}(d;s) =\sum_{f\in \mathcal C(d,\x,L,J)}(-1)^{m_{L,F}(f)} $$
only depends on the choices of $L$, $F$, $d$,  $s$, and on the
deformation class of $X_\R$ \cite{Wel1,Wel11}, and is called a genus 0
\emph{Welschinger
invariant} of $X_\R$.
Following \cite{IKS11},
one can generalize the
previous definition of Welschinger
invariants to  any class $F\in
H_2^{\tau_X}(X\setminus L;\Z/2\Z)$,
see  Section \ref{sec:abs welsch}. Note nevertheless that among all
 possible choices of $F$ in $H_2^{\tau_X}(X\setminus 
L;\Z/2\Z)$, the two classes $0$ and $[\R X\setminus L]$ seem to 
play a special role,
see Remark \ref{rem:two special F}.
Hence it seems worthwhile to specialize Theorem \ref{thm:main1}
 in these two special cases.

\begin{thm}\label{thm:main1b} 
Let  $X_\R$ be a compact real symplectic manifold of dimension 4. 
Let $S$ be a real Lagrangian sphere in $X_\R$, and  $L$ be 
a connected components of $\R X$  disjoint from
$S$.
We denote by $Y_\R$ the surgery of $X_\R$ along
$S$, and we assume that  $\chi(\R Y)=\chi(\R X)+ 2$.

Then  for any class 
 $d\in H^{-\tau_Y}_2(X;\Z)$, the two following identities hold:
\begin{align*}
 &  W_{Y_\R,L,0}(d;s) = W_{X_\R,L,0}(d;s) + 2\sum_{k\ge
  1}\ W_{X_\R,L,0}(d-k[S];s),
  \\ \mbox{and} \qquad\qquad\qquad& 
  \\ &W_{Y_\R,L,[\R Y\setminus L]}(d;s) = W_{X_\R,L,[\R X\setminus
    L]}(d;s) + 2\sum_{k\ge 1}\ (-1)^k\ W_{X_\R,L,[\R X\setminus
    L]}(d-k[S];s).
\end{align*}
\end{thm}
The case when $F=0$ follows immediately from  Theorem \ref{thm:main1}.
The case when $F=[\R X\setminus L]$ follows  from  Theorem
\ref{thm:main1} combined with
the identity $[\R Y\setminus L]=[\R X\setminus L] +[S]$ in $H_2^{\tau_Y}(X\setminus
L;\Z/2\Z)$.

As mentioned in the beginning, one may reasonably expect that a suitable combination of Theorem
 \ref{thm:main1}  with Solomon's real WDVV equations \cite{HorSol12,Sol1} reduces
the computation of genus 0 Welschinger invariants of all real
rational algebraic
surfaces to the cases of $\CP^2$ and $\CP^1\times\CP^1$. (Recall that thanks to complex WDVV
 equations \cite{KonMan1},
 the computation of genus 0 Gromov-Witten invariants of
 all rational symplectic 4-manifolds can be reduced to  computations
 in $\CP^2$ and $\CP^1\times\CP^1$ \cite{PanGot98,McD90}.)

\begin{rem}\label{rem:better proof}
Although the formulas from Theorem \ref{thm:main1b} are surprisingly simple, our
proof goes by tedious computations involving binomial
coefficients. There might exists a simpler and more transparent (geometrical) proof of
Theorems \ref{thm:main1b} and \ref{thm:main1}. 
\end{rem}
\begin{rem}
  Given a Lagrangian sphere $S$  in a symplectic manifold
  $(X,\omega)$, one can define  a symplectomorphism of $(X,\omega)$
  called a
  \emph{Dehn twist} along $S$, see \cite{Arn95}. If $X$ is four
  dimensional, this automorphism acts on $H_2(X;\Z)$ by the involution
  $d\mapsto d+(d\cdot [S])[S]$.
  Hence similarly to the complex setting, if $S$ is a real Lagrangian
  sphere in a real symplectic fourfold 
  $X_\R$,  one thus obtains the following relation for Welschinger invariants:
\begin{equation}\label{equ:dehn}
W_{X_\R,L,F}(d;s)= W_{X_\R,L,F}(d+(d\cdot [S])[S];s) \qquad\forall d\in H^{-\tau_X}_2(X;\Z).
\end{equation}
(Note that by Lemma \ref{lem:sphere invariant}, this relation can be
non-trivial only when 
$[S]\in H^{-\tau_X}_2(X;\Z)$.) Equation $(\ref{equ:dehn})$ allows to
rewrite the two identities from Theorem \ref{thm:main1b} in a more
symmetric form\footnote{These two formulas also independently appeared in the
February 2017 arXiv version of \cite[Corollary 4.3]{IKS13}.}
\begin{align*}
 &  W_{Y_\R,L,0}(d;s) = \sum_{k\in\Z}\ W_{X_\R,L,0}(d+k[S];s),
  \\ &W_{Y_\R,L,[\R Y\setminus L]}(d;s) = \sum_{k\in\Z}\ (-1)^k\ W_{X_\R,L,[\R X\setminus
      L]}(d+k[S];s),
\end{align*}
which may be useful in the perspective of Remark
\ref{rem:better proof}.
\end{rem}

\begin{rem}
Several formulas involving Welschinger invariants, especially those
based on degeneration formulas like symplectic sum formulas, are 
real counterparts of analogous formulas relating Gromov-Witten
invariants, see for example \cite{Mik1,IKS3,IKS13,Br7,Br6b,Br20,BP14,Bru14,BruGeo15}. This led Göttsche to conjecture
the existence of \emph{quantum enumerative invariants} that would in
particular 
contain both Gromov-Witten and Welschinger invariants as
suitable specializations, see for example  \cite{GotShe12,BlGo14,IteMik13,Mik15}.
Theorems \ref{thm:main1b} and \ref{thm:main1} might have an interpretation in this perspective.
\end{rem}

\medskip
An immediate consequence of Theorem \ref{thm:main1b}
is that
positivity and asymptotic results concerning Welschinger invariants
of $X_\R$
transfer to
$Y_\R$. Particular instances of such positivity and asymptotic results
can be found in \cite{IKS2,IKS10,IKS11,IKS13,Shu14,Br7,Br6b,Br6,Bru14}. 
\begin{cor}
Let $X_\R$, $Y_\R$, $S$, and $L$ be as in Theorem \ref{thm:main1b}. If 
$ W_{X_\R,L,0}(d;0)\ge 0$ for any $d\in H_2^{-\tau_X}(X;\Z/2\Z)$, then
$W_{Y_\R,L,0}(d;0)\ge 0$  for any $d\in H_2^{-\tau_Y}(X;\Z/2\Z)$.

If furthermore for a given $d\in H_2^{-\tau_Y}(X;\Z/2\Z)$, 
the sequence $(W_{X_\R,L,0}(nd;0))_{n\ge 1}$ is
logarithmically asymptotic to the sequence of the corresponding Gromov-Witten
invariants of $(X,\omega_X)$, then so is the sequence  $(W_{Y_\R,L,0}(nd;0))_{n\ge 1}$.
\end{cor}

Theorems \ref{thm:main1b} and \ref{thm:main1} have several
applications to the case of real rational algebraic surfaces, see Sections
\ref{sec:rational surf} and \ref{sec:concrete}. In particular,
 combined with results
from \cite{Bru14}, they  complete the computation of genus 0 Welschinger invariants of
all real del Pezzo surfaces.

\bigskip
\noindent{\bf Context and relation to other works.}
Using a real version of the symplectic sum formula (we refer for example to
\cite{IP,LiRu01,EGH} for complex versions, see also\cite{Li02,Li04}
for an analogous formula in the algebraic category), we related 
in \cite{Br20,BP14} genus 0  Welschinger invariants  of two real symplectic $4$-manifolds
$X_\R$ and $Y_\R$
differing by a  surgery along a real Lagrangian sphere $S$ (see also
\cite{IKS13,Shu14} for  related works in the case of algebraic del~Pezzo
surfaces).
In general, relations from
\cite{BP14}
involve some quantities that
depend on some choices additional to the choice of $X_\R$, $Y_\R$, and
$S$. These quantities come from  the enumeration of $J$-holomorphic
curves in a real deformation of either $ X_\R$ of $Y_\R$ 
for which
$S$ becomes symplectic, and with $J$ chosen so that $S$ is
$J$-holomorphic. Since $S$ has self-intersection $-2$, this almost
complex structure $J$ is not generic enough to ensure that counting
real $J$-holomorphic curves with Welschinger signs give rise to an invariant.
Still, relations from \cite{BP14} have been applied in
\cite{BP14,Bru14}
 to obtain qualitative results and explicit computations
 of Welschinger invariants in a number of cases (see also the related
 works \cite{IKS13,Shu14} in the case of algebraic del~Pezzo surfaces of degree at
 least 2).

There are nevertheless particular situations where 
relations  from 
\cite{BP14} simplify so that only enumerative invariants of $X_\R$ and
$Y_\R$ remain.
 Recall 
 that the definition of 
genus 0 Welschinger invariants requires the choice of
 a connected component $L$ of  $\R X$, and of a class $F\in
 H_2^{\tau_X}(X\setminus L;\Z/2\Z)$.
In the case when  $L$ is disjoint from the real Lagrangian sphere $S$,
and  $F$ is orthogonal to $[S]$, then
\cite[Theorem 2.5(1)]{BP14} ultimately only involves  genus 0 Welschinger invariants
of $ X_\R$ and $Y_\R$. Although this is an obvious consequence of the results
exposed in \cite{BP14},
this remark is not explicitly made there (see for example
\cite[Corollaries 4.2 and 4.3]{IKS13} for  explicit similar
remarks in the case of
algebraic del
Pezzo surfaces).

The aim of the present paper is to make explicit and to provide several
applications  of
relations among   Welschinger invariants
of $ X_\R$ and $Y_\R$. 
As a by-product, we obtain the existence of
some new relative Welschinger invariants, see Theorem \ref{thm:relative}.
Note that very few relative invariants are known in real enumerative
geometry so far, see for example \cite{Wel5,Wel4,IKS13,Shu15,IKS16}. 
The existence of absolute Welschinger invariants and 
Theorems \ref{thm:relations 1} and \ref{thm:relations 2}
immediately imply the existence of real
invariants relative to some collections of disjoint real embedded symplectic spheres
with self-intersection $-2$, where only simple and non-fixed
incidences to these spheres are prescribed.
Furthermore we provide in Theorems
\ref{thm:relations 1} and \ref{thm:relations 2}
a mild generalization of
\cite[Theorem 2.5]{BP14} to the case of hypothetical Welschinger invariants in
positive genus defined in the same vein as in \cite{Shu14}.

\medskip
Welschinger first proposed  in  \cite{Wel4} 
an other, nevertheless
related, treatment of  Lagrangian spheres contained in $\R X$.
 Among other results, he proved 
there that   all invariants  $W_{X_\R,S^2,F}$ can be expressed
in terms of some relative Gromov-Witten invariants of the symplectic
manifold $(X,\omega_X)$, and some
real invariants of
$T^*S^2$. 

Several real   enumerative invariants of higher dimensional
symplectic manifolds have
been defined in the last fifteen years, i.e. \cite{Wel2,Wel3,Geor13,GeoZin15}.
It could be interesting to generalize the methods of the present
paper to the study of these invariants.

\bigskip
\noindent{\bf Organization of the paper.}
 In Section \ref{sec:real surg} we
describe in detail  surgeries of real symplectic $4$-manifolds
along real Lagrangian spheres, and give several examples of such
surgeries. 
Absolute and relative 
Welschinger invariants considered in this paper 
are defined in Section \ref{sec:welsch}. We prove there
 Theorems \ref{thm:relations 1}
and \ref{thm:relations 2} that relate such invariants for two  real
symplectic $4$-manifolds differing by a  surgery along a real
Lagrangian sphere, from which we deduce Theorem \ref{thm:main1}.
Qualitative applications of
Section \ref{sec:welsch} to the case of real rational algebraic
surfaces  are discussed in Section 
\ref{sec:rational surf}.
Finally, we illustrate in Section \ref{sec:concrete} 
the use of Theorem \ref{thm:main1} with
concrete computations in the case of real cubic surfaces, $\R$-minimal real
conic bundles, and real del~Pezzo surfaces of degree 1. 

\bigskip
\noindent{\bf Acknowledgment.} I am grateful to
Benoît B.  Bertrand, Nicolas Puignau, as well as to anonymous referees for
their many valuable
comments on earlier versions of this paper, and  to Yanqiao Ding
for pointing me a few misprints. 
I am also indebted to
Jean-Yves Welschinger and Vincent Colin for discussions that helped me to precise
several aspects of the work presented here. This  work is
partially supported by the grant TROPICOUNT of Région Pays de la Loire.

\section{Surgery  along a
 real Lagrangian sphere}\label{sec:real surg}

\subsection{Real structures on a quadric surface}\label{sec:sphere}
Here we recall some well-known facts about projective and affine
quadrics. Any such quadric is always assumed to be equipped with
the symplectic form $\omega_{FS}$ induced by the restriction  of the
Fubini-Study form on $\CP^3$.

A non-singular complex algebraic quadric surface $Q$ in
$\CP^3$ is biholomorphic to $\CP^1\times\CP^1$. In particular the
group $H_2(Q;\Z)$ is isomorphic to $\Z^2$ and generated by the classes
$l_1=[\CP^1\times\{p\}]$ and $l_2=[\{p\}\times\CP^1]$. Clearly, these
two classes are well defined in $H_2(Q;\Z)$ only up to interchanging
$l_1$ and $l_2$. A hyperplane section $E$ of $Q$ realizes the class
$l_1+l_2$. In what follows $E$ is always assumed to be  non-singular,
which implies in particular that it is biholomorphic  to $\CP^1$.
In a suitable coordinate system,
the complement $Q \setminus E$ is given in the corresponding affine
chart of $\CP^3$ by the equation
\[
x^2+y^2+z^2=1.
\]
When in addition both $Q$ and $E$ are real (i.e. stable under the
standard complex conjugation on $\CP^3$),
the  affine quadric $Q \setminus E$ is given by the following equation
in  a suitable real coordinate system
\begin{equation}\label{equ:quad}
  (-1)^{\epsilon_1} x^2+(-1)^{\epsilon_2} y^2+(-1)^{\epsilon_3} z^2=1
  \qquad \mbox{with }\epsilon_i\in\{0,1\}.
  \end{equation}
The trace  $S_Q$ of $Q$ on   
$i^{\epsilon_1}\R\times i^{\epsilon_2}\R \times i^{\epsilon_3}\R $
is  the unit 2-sphere, and is real  Lagrangian in $Q\setminus E$.
Furthermore, it  realizes the class $\pm(l_1-l_2)\in H_2(Q;\Z)$ when endowed with some orientation.
Different choices of $\epsilon_1,\epsilon_2$, and $\epsilon_3$ provide 
four different real structures on the pair
$(Q,E)$, see Figure \ref{fig:real quad}:
\begin{itemize}
\item $\tau_{S^1,0}$: $\R E\ne\emptyset$, and $\R (Q\setminus E)$
  is a one-sheeted hyperboloid;
\item $\tau_{S^1,2}$: $\R E\ne\emptyset$, and $\R (Q\setminus E)$ is a two-sheeted hyperboloid;
\item $\tau_{\emptyset,0}$: $\R E=\emptyset$, and $\R Q=\emptyset$;
\item $\tau_{\emptyset,2}$: $\R E=\emptyset$, and $\R Q=S^2$ is an ellipsoid.
\end{itemize}
\begin{figure}[h!]
\begin{center}
\begin{tabular}{ccccccc}
\includegraphics[width=3cm, angle=0]{Figures/H1_hyp.pdf}
\put(-5, 110){$\R (Q\setminus E)$}
\put(-50, 45){$S_Q$}
& \hspace{8ex} &
\includegraphics[width=3cm, angle=0]{Figures/H1_el2.pdf}
\put(-7, 130){$\R  (Q\setminus E)$}
\put(-7, 40){$\R  (Q\setminus E)$}
\put(-10, 80){$S_Q$}
& \hspace{5ex} &
\includegraphics[width=3cm, angle=0]{Figures/H1_em.pdf}
\put(-50, 45){$S_Q$}
& \hspace{0ex} &
\includegraphics[width=3cm, angle=0]{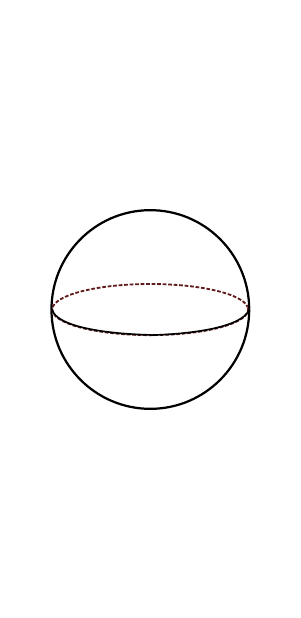}
\put(-65, 45){$\R Q=S_Q$}
\\ \\ $x^2+y^2-z^2=1$ &&  $-x^2-y^2+z^2=1$ &&  $-x^2-y^2-z^2=1$ && $x^2+y^2+z^2=1$ 
\\ \\a) $\tau_{S^1,0}$ && b) $\tau_{S^1,2}$
&& c)
$\tau_{\emptyset,0}$ 
&&d) $\tau_{\emptyset,2}$
\end{tabular}
\end{center}
\caption{Real structures on $(Q,E)$}
\label{fig:real quad}
\end{figure}
For each real structure
$ \tau_{A,a}$ we have
$$\R E=A\qquad\mbox{and}\qquad \chi(\R Q)=a. $$
Note that the two real structures $\tau_{A,0}$ and  $\tau_{A,2}$ on
the pair $(Q,E)$
differ one from the other by the composition with
the automorphism $(x,y,z)\mapsto (-x,-y,-z)$ of $\C^3$, i.e.
the two corresponding equations of the form $(\ref{equ:quad})$ are
obtained one from the other by the change of variable
$(x,y,z)\mapsto (ix,iy,iz)$.

\subsection{Surgery along a real Lagrangian
  sphere}\label{sec:def real surgery}
Let $X_\R$ be a real  symplectic
$4$-manifold containing a real Lagrangian sphere $S$.
Recall that this means that $S$ is globally invariant under $\tau_X$.
It is proved in \cite[Proposition 2.1 and Lemma 2.14]{Teh10} that $X_\R$ is deformation equivalent to
the equivariant symplectic sum 
of  two real symplectic $4$-manifolds $Z_\R=(Z,\omega_Z,\tau_Z)$
and a real quadric $(Q, \omega_{FS},\tau_{A,a})$ along
an embedded symplectic  
sphere $E$ of self-intersection
$-2$ in $Z$ (hence of self-intersection
$2$ in $Q$) where:
\begin{itemize}
\item $Z_\R$ 
  is a symplectic reduction of $X$ with a small neighborhood of $S$ removed;
\item $E$ is a real hyperplane section of $Q$;
\item $S$ is a Lagrangian deformation of a  real Lagrangian sphere
  $S_Q$ in
  $(Q\setminus E,\omega_{FS},\tau_{A,a})$.
\end{itemize}

\begin{rem}
Audin proposed in  \cite{Aud07} an earlier 
non-equivariant version of this construction. 
Welschinger  proposed in \cite{Wel4} an equivalent approach using symplectic
field theory. 
\end{rem}

From an algebraic geometric perspective,
the degeneration of $X_\R$ to the union of $Z_\R$ and
$(Q, \omega_{FS},\tau_{A,a})$ 
 can be thought as a
 degeneration of $X_\R$ to a real nodal 
symplectic manifold for which $[S]$ is precisely the vanishing cycle, 
see Example \ref{ex:gluing real part} below. From this point of view, one
may think of $Z_\R$ as  obtained from $X_\R$ by contracting $S$ to a point and then blowin-up the
resulting ordinary double point.

\begin{defi}
  We say that a real symplectic manifold $Y_\R$ is a
  \emph{surgery of $X_\R$ along the real
    Lagrangian sphere $S$} if it is deformation equivalent to
  the equivariant symplectic sum of
 $Z_\R=(Z,\omega_Z,\tau_Z)$
and $(Q, \omega_{FS},\tau_{A,2-a})$.
\end{defi}
Phrased differently,  the real
manifold $Y_\R$ is obtained from $X_\R$ by changing the symplectic and
real structures
of $X_\R$ in a neighborhood of $S$. Since  all surgeries of $X_\R$ along the real
    Lagrangian sphere $S$ are deformation equivalent, and since we are
    interested in properties that are invariants under deformation, we
    say that $Y_\R$ is \emph{the} surgery of $X_\R$ along the real
    Lagrangian sphere $S$ rather than \emph{a} surgery.
By extension, we will also say that $Y_\R$
  is obtained from $Z_\R$ by a  surgery along $E$.
The notation  $Y_\R \xrightarrow[]{S}  X_\R$ 
means that $X_\R$ and $Y_\R$ are related
by a  surgery along  the real Lagrangian sphere $S$, and that
$\chi(\R Y)=\chi(\R X)+2$.
In this case, the real part $\R Y$ is obtained from $\R X$ by one of the
following topological operations:
\begin{itemize}
\item if $A=S^1$: cut $\R X$ along $\R S=S^1$ and glue a disk to each
  boundary circle (from Figure \ref{fig:real quad}a to  Figure
  \ref{fig:real quad}b);
\item if $A=\emptyset$: the sphere $S$ which contains no real point
  of $\R X$ becomes a connected component of $\R Y$ (from Figure \ref{fig:real quad}c to  Figure
  \ref{fig:real quad}d).
\end{itemize}
Note that both symplectic manifolds $(X,\omega_X)$ and $(Y,\omega_Y)$
are a
deformation of $(Z,\omega_Z)$. 
However the real manifold $X_\R$ is
 a real deformation of $Z_\R$ if and only if 
 $\chi(\R X)=\chi(\R Y)-2$.

\begin{exa}
A real quadric hyperboloid in $\CP^3$ is obtained from a real quadric
ellipsoid by a surgery along a real Lagrangian sphere
intersecting the real part in two points (see Figures \ref{fig:real quad}a and
  \ref{fig:real quad}b).
A real empty quadric in $\CP^3$ is obtained from a real quadric
ellipsoid by a  surgery along the real part (see Figures \ref{fig:real quad}c and
  \ref{fig:real quad}d).
\end{exa}

\begin{exa}\label{ex:gluing real part}
More generally, 
let $D\subset\C$ be a small disk endowed with its standard real structure, 
and let $\pi:\mathcal X \to D$ be a flat real
morphism from a non-singular real algebraic manifold of complex
dimension 3.
 Suppose that the fiber $X_t=\pi^{-1}(t)$ is a
 non-singular  projective algebraic surface when $t\ne 0$, and is 
 a   real  algebraic surface with a single 
non-degenerate double
point $p$ as  only singularity when $t=0$. Suppose in addition that
$\mathcal X$ is locally given at $p$ by the equation
 $$x^2 + y^2\pm z^2=t \qquad (x,y,z,t)\in\C^4,$$
 and that $\pi$ is locally given by 
 $\pi(x,y,z,t)=t$.
Let us perform respectively the  base changes $t\to
 t^2$ and $t\to - t^2$. Then the blow-up at the node of the two obtained families
 realize  the two  symplectic sums
described above. In particular the real algebraic surface $Z_\R$ is simply the blow-up of the
singular surface $X_0$ at the node.
Hence $X_t$ and $X_{-t}$ are obtained one from the other by a
surgery along  a real Lagrangian sphere $S$ realizing
the vanishing cycle of the degeneration of $X_t$ to
$X_0$. 
\end{exa}
\begin{exa}\label{ex:cubic surface}
Let $X_\R$ be a non-singular real cubic surface  in $\CP^3$
with a real part consisting of the disjoint union of a projective
plane $\RP^2$ and a sphere $S$. It is classical that the underlying complex
algebraic surface is the complex projective plane $\CP^2$ blown up at $6$
points in general position (see for example \cite{Dol12}). Note
however that since $\R
X$ has two connected components, the real surface $X_\R$ is not
rational over $\R$, and is not obtained as a  blow-up of the real
projective plane.

By Example \ref{ex:gluing real part},
the surgery of $X_\R$ along $S$ is a real algebraic cubic surface
whose real part is homeomorphic to $\R P^2$, and it is classical that it is  the  blow-up
of  the real projective plane 
 at three pairs of complex conjugated points (see for example \cite{Man86,Kol97,DegKha02}). 
\end{exa}

\begin{lemma}\label{lem:sphere invariant}
Suppose that $Y_\R \xrightarrow[]{S}  X_\R$. Then the class $[S]$ is in $H_2^{-\tau_X}(X;\Z)$ and in
$H_2^{\tau_Y}(X;\Z)$.
Furthermore we have
$$H_2^{-\tau_Y}(X;\Z)\subset [S]^\perp \qquad \mbox{and}\qquad 
 H_2^{-\tau_X}(X;\mathbb Q)=H_2^{-\tau_Y}(X;\mathbb Q)\oplus\mathbb Q[S].$$
\end{lemma}
\begin{proof}
The first claim follows from the fact that the analogous statement holds for
affine quadrics.
Let $\gamma\in H_2^{-\tau_Y}(X;\Z)$. 
Since $[S]^2=-2$, the sphere $S$ realizes a non-trivial class in 
$H_2(X;\mathbb Q)$ and we have
$$\gamma=\gamma_\perp+q [S] \quad\mbox{in }  H_2(X;\mathbb Q) \quad
\mbox{with } \gamma_\perp\in [S]^\perp \mbox{ and }q\in\mathbb Q.$$
Since $\tau_{Y *}(\gamma)=-\gamma$, we have
$$2q= -q[S]\cdot[S]=-\gamma\cdot [S]=\tau_{Y*}(\gamma)\cdot [S]=q[S]\cdot
[S]=-2q, $$
from which we deduce that $q=0$, that is to say $\gamma\in
[S]^\perp$. Since the restrictions of $\tau_{X*}$ and $\tau_{Y*}$
coincide on $[S]^\perp$, we obtain
that
$H_2^{-\tau_X}(X;\mathbb Q)=H_2^{-\tau_Y}(X;\mathbb Q)\oplus\mathbb Q[S].$
\end{proof}

Since the first Chern class of $X$ is $\tau_X$-anti-invariant for any
real structure $\tau_X$ on
$(X,\omega_X)$, we have in particular 
$$c_1(X)\cdot [S]=0. $$
(this also follows from the adjunction formula.)  

Considering homology with coefficients in $\Z/2\Z$, the previous lemma
can be weakened as follows.
\begin{lemma}\label{lem:2 real class}
Suppose that $Y_\R \xrightarrow[]{S}  X_\R$. Then we have
$$H_2^{\tau_X}(X;\Z/2\Z)\cap [S]^\perp = H_2^{\tau_Y}(X;\Z/2\Z )\cap [S]^\perp. $$
\end{lemma}

\section{Welschinger invariants}\label{sec:welsch}

In this section we define 
absolute and relative 
Welschinger invariants considered in this paper, and we prove
 Theorems \ref{thm:relations 1}
and \ref{thm:relations 2} that relate such invariants for two  real
symplectic $4$-manifolds differing by a  surgery along a real
Lagrangian sphere. Welschinger invariants of symplectic $4$-manifolds
are up to now only defined in the case of rational curves, nevertheless
Shustin proposed in \cite{Shu14} a partial generalisation to positive genus in the case of
algebraic del Pezzo surfaces. 
Our proof of  Theorems \ref{thm:relations 1}
and \ref{thm:relations 2} extends to enumeration of curves of any
genus, hence we decided to state both theorems for  (hypothetical if $g>0$)
Welschinger invariants of any genus defined in the same vein as in \cite{Shu14}.
The proof from  \cite{Shu14} should be adaptable to the symplectic setting
 in the
 obvious way  using the strategy proposed in \cite{Wel1}
 (including the correction from \cite{Wel11}).
 Doing so would
 nevertheless bring us quite far from our original purposes, so we
 leave the existence of 
 of Welschinger invariants of positive genus
 considered in this text as an hypothesis.

\subsection{Preliminaries}\label{sec:preliminaries}

We start by proving a simple adaptation of 
\cite[Lemma 3.1 and Proposition 3.3]{BP14} that we will use at several
places in the rest of
this section.

Let $(X,\omega_X)$ be a compact symplectic manifold of dimension 4,
containing a finite union $W= E_1\cup \ldots\cup E_\kappa$
of pairwise disjoint
embedded symplectic spheres
with $[E_i]^2=-2$.
Let also 
$J$ be  an almost complex structure on $X$ tamed by $\omega_{X}$
 for which all curves $E_1,\ldots,E_\kappa$ are
$J$-holomorphic. It is classical that such $J$ exists, see for example \cite[Propposition 2.2]{Wen18}.
Given $d \in H_2(X;\Z)$ and $g\in\Z_{\ge 0}$,
let us choose a configuration $\x$ of  $c_1(X)\cdot d+g-1$ distinct
points in $X\setminus \displaystyle \bigcup_{i=1}^\kappa E_i$. 
We define
$\CC^\C(d,g,\x,W,J)$ as the set of 
irreducible $J$-holomorphic curves $f:C\to X$ of
genus $g$, with  $f_*[C]=d$, passing through all points in
$\x$, and whose image is not contained in $\displaystyle \bigcup_{i=1}^\kappa E_i$. 
Such a $J$-holomorphic curve $f:C\to X$ is said to be \emph{nodal} if
all singularities of $f(C)$, if any, are transverse self-intersections.

\begin{lemma}\label{lem:finite relative}
Suppose that $c_1(X)\cdot d>0$ if $g=1$.
Then for a generic choice of $J$ among  almost
complex structure $J$ tamed by $\omega_X$ such that all 
symplectic curves $E_1,\ldots,E_\kappa$  are $J$-holomorphic,
 the set  $\mathcal C^\C(d,g,\x,W,J)$ is finite
and composed of simple maps that are all nodal immersions. 
\end{lemma}
\begin{proof}
The proof consists in two steps: first we prove that no element of 
$\mathcal C^\C(d,g,\x,W,J)$ factors through a  non-trivial ramified covering, from which
we  deduce the finiteness of $\mathcal C^\C(d,g,\x,W,J)$.
Then all maps in $\mathcal C^\C(d,g,\x,W,J)$
are nodal immersions by
\cite[Corollaries 2.26, 2.30, and 2.32, and Remark 2.17]{Wen18}.
By perturbing $J$ in the complement of a small neighborhood of $W$ if
necessary, we may assume that  for any class $d_0\in H_2(X;\Z)$ and
any $J$-holomorphic simple map $f_0:C_0\to X$ such that
$f_{0 *}[C_0]=d_0$, the curve
$C_0$ has
genus $g_0$, and $\x\subset f(C_0)$, we have
\begin{equation}\label{equ:generic}
c_1(X)\cdot d_0 +g_0 -1 \ge  c_1(X)\cdot d +g -1,
\end{equation}
see for example \cite[Corollary 2.23 and Remark 2.17]{Wen18}.

{\bf Step 1.} Let $f:C\to X$ be an element of $\mathcal C^\C(d,g,\x,W,J)$
that factors through a ramified covering of degree $\delta\ge 2$ 
of a simple map $f_0:C_0\to X$.
Denoting by $g_0$ the genus of $C_0$, we obtain by the Riemann-Hurwitz
formula that
\begin{equation}\label{equ:RH}
g\ge \delta g_0 +1 - \delta.
\end{equation}
Let $d_0\in H_2(X;\Z)$ denotes the class $f_{0 *}[C_0]$.
Since $d=\delta d_0$, Inequality $(\ref{equ:generic})$ becomes
$$(\delta-1)c_1(X)\cdot d_0 +g-g_0\le 0.  $$
Combining this with $(\ref{equ:RH})$, we obtain
$$(\delta-1)(c_1(X)\cdot d_0 +g_0 -1)\le 0, $$
and so 
$$c_1(X)\cdot d_0 +g_0-1\le 0. $$
By \cite[Corollary 2.23 and Remark 2.17]{Wen18}, 
we also have the
opposite inequality,
which alltogether gives 
\begin{equation}\label{equ:factor}
c_1(X)\cdot d_0+g_0 -1= 0. 
\end{equation}
Furthermore, all inequalities above are in fact equalities. In
particular, the covering $C\to C_0$ through which $f$ factors is
non-ramified, which is possible only if $g=g_0=1$. In this case
$(\ref{equ:factor})$ gives $c_1(X)\cdot d=c_1(X)\cdot d_0=0$ which is
excluded by assumption.

\medskip
{\bf Step 2.} 
Suppose that $\CC^\C(d,g,\x,W,J)$ contains infinitely many
simple maps. By Gromov compactness Theorem, there exists a sequence
$(f_n)_{n\ge 0}$ of distinct simple maps in $\CC^\C(d,g,\x,W,J)$ which
converges to some $J$-holomorphic map
$\overline f:\overline C \to X$.
The genericity of $J$ implies that 
the set of simple maps in  $\CC^\C(d,g,\x,W,J)$ is
a $0$-dimensional manifold, and in particular is
discrete (see \cite[Theorem 2.21 and Remark 2.17]{Wen18}). Hence
 either $\overline C$ is reducible, or $\overline f$ is non-simple.
Let 
$\overline C_1,\ldots , \overline C_m, $ $\overline C'_1,\ldots,\overline C'_{m'}$ 
be the irreducible components of $\overline C$, labeled in such a way that
\begin{itemize}
\item  $\overline  f(\overline C_i)\not\subset W$ for any
  $i\in\{1,\cdots , m\}$;
\item $\overline f(\overline C'_i)\subset W$, and $\overline
  f_{|\overline C'_i}$ factors through a ramified covering of
  degree $k_i$, for any
  $i\in\{1,\cdots , m'\}$.
\end{itemize}
Define $k=\sum_{i=1}^{m'}k_i$, and denote by $g_i$ the genus of $C_i$.
The 
restriction of $\overline f$ to
$\displaystyle \bigcup_{i=1}^m\overline C_i$ is subject to $c_1(X)\cdot d +g -1$ 
points conditions, so we have
\begin{equation}\label{equ:reducible}
c_1(X)\cdot d_1+\sum_{i=1}^mg_i  -m\ge c_1(X)\cdot d +g -1,
\end{equation}
where $d_1$ is the class realized by the image of this restriction.
Since an irreducible component $E$ of $W$
 is an embedded sphere with self-intersection $-2$, the adjunction formula implies  that 
$c_1(X)\cdot [E]=0$. Hence we get $c_1(X)\cdot d_1=c_1(X)\cdot
 d$. Combining with $(\ref{equ:reducible})$ we obtain
$$\sum_{i=1}^mg_i -g +1-m\ge 0. $$
If $m=0$, then the image of $\overline f$ must be contained in a curve
in $W$, implying that $c_1(X)\cdot d=0$. Since furthermore in this
case we would have  $c_1(X)\cdot d+g-1=0$, we deduce that $g=1$ contrary to our
assumptions.
Hence $m\ge 1$, and
since $\sum_{i=1}^mg_i \le g$, we deduce that $m=1$ and $g_1=g$. In
particular, each irreducible component $\overline C'_i$ of $\overline
C$ is rational and intersects $\overline C_1$ in a single point.

If $k=m'=0$, then the curve $\overline C$ is irreducible. 
Hence as explained above, the map $\overline f$ has to factorize through 
 a non-trivial ramified covering of a simple
map $f_0:C_0\to X$, which contradicts Step 1.
Hence we have $k>0$.
By genericity of $J$, 
the curve $\overline f(\overline C_1)$ is fixed by the 
$c_1(X)\cdot d+g-1$ point constraints in $X$.
By perturbing $J$ in a neighborhood of $W$ if necessary, we may
assume that   $\overline f(\overline C_1)$ intersects the
curve $W$ transversely. 
Any intersection point of $\overline f(\overline
C_1\setminus (\overline C'_1\cup\ldots\cup \overline C'_{m'}))$ and
 $W$ 
 deforms to
an intersection point of the image of $f_n$ and  $W$ 
for $n>>1$.
Since
$d_1\cdot [W]=d\cdot [W] +2k$ and $m'\le k$, 
at least $d\cdot [W] +k$ intersection points of $\overline f(\overline
C_1)$ and $W$ deform to an intersection point of the image of $f_n$ and $W$
for $n>>1$. But this contradicts the fact that two $J$-holomorphic
curves intersect positively.
\end{proof}

\begin{rem}\label{rem:real g}
Lemma \ref{lem:finite relative} has an obvious equivariant version
when $(X,\omega)$ is equipped with an anti-symplectic involution
$\tau$ for which 
$W$ is $\tau$-anti-invariant. The
proof is by adapting the proof of Lemma \ref{lem:finite relative}
following the proof of  \cite[Theorem 1.10]{Wel1} in the genus 0 case.
\end{rem}

\subsection{Absolute Welschinger invariants}\label{sec:abs welsch}
Recall that if $C$ is an irreducible compact non-singular real algebraic curve of genus
$g$, then the set $\R C$ has at
most $g+1$ connected components by the Harnack-Klein inequality. The
real curve $C$ is called 
\emph{maximal} when equality holds. In this case, the set 
$C\setminus \R C$ has two connected components.
The following lemma is an immediate consequence of \cite[Lemme 3.6.22]{Mang17} and the
Smith exact sequence.
\begin{lemma}\label{lem:class real part}
Let $X_\R$ be a connected real symplectic $4$-manifold with
$b_1(X;\Z/2\Z)=0$ and $\R X\ne\emptyset$,
and
let $\Gamma\in H_2(X,\R X;\Z/2\Z)$. Then the image of
$\partial \Gamma\in H_1(\R X;\Z/2\Z)$ only depends on the class 
$\Gamma+\tau_{X,*}(\Gamma)\in H_2(X;\Z/2\Z)$.
In particular, any class $d\in H_2^{\tau_X}(X;\Z/2\Z)$ induces a class 
$l_{d}\in H_1(\R X;\Z/2\Z)$.
\end{lemma}
Given a connected component $L$ of $\R X$, we denote  by
$l_{L,d}$
the natural projection of the class $l_{d}$
to $H_1(L;\Z/2\Z)$.
  
\medskip
For the rest of this section we fix once for all an integer $g\ge 0$, and
 $X_\R=(X,\omega_X,\tau_X)$  a real compact symplectic manifold of dimension
4. If $g>0$, we furthermore assume that $b_1(X;\Z/2\Z)=0$.

\medskip
Suppose that $\R X$  contains $g+1$ connected components
denoted by 
$L_1,\ldots,L_{g+1}$, and define $L=\displaystyle \bigcup_{i=1}^{g+1}L_i$.
 Note that $\R X$ might contain other connected
components. 

\medskip
We say that 
 an  almost complex structure $J$ tamed by $\omega_X$ is
 $\tau_X$-compatible if $\tau_X$ is $J$-antiholomorphic, i.e. $J\circ
 d \tau_X=-d\tau_X\circ J$.
Recall that there exists a well defined pairing 
$$H_2(X,L;\Z/2\Z)\times H_2(X\setminus L;\Z/2\Z)\to\Z/2\Z$$
given by the intersection product modulo 2.
Let $C$ be a maximal irreducible real algebraic curve,
and  $f:C\to X$ be a real $J$-holomorphic nodal immersion such that $f(\R
C)\subset L$, for some $\tau_X$-compatible almost 
complex structure $J$ on $X$. 
Denoting by $C^+$ the topological closure of one of the halves of $C\setminus\R C$,
and given  $F\in H^{\tau_X}_2(X\setminus L;\Z/2\Z)$,
 we
define the
 $(L,F)$-mass of $f$ as
$$m_{L,F}(f)=m(f)+[f(C^+)]\cdot F, $$
where $m(f)$ is the number of  elliptic real  nodes of  $f(C)$ (i.e.
real nodes with two
 $\tau_X$-conjugated branches) contained in $L$. Note that $m_{L,F}(f)$ does not depend on
the chosen  half of $C\setminus\R C$.

\begin{exa}
If $F=[\R X\setminus L]$,  then $m_{L,F}(f)$ is 
the total number of elliptic real  nodes of  $f(C)$.
\end{exa}

Choose   a class $d\in
H_2^{-\tau_X}(X;\Z)$,
and  $\underline r=(r_1,\ldots r_{g+1})\in\Z_{\ge 0}^{g+1}$ and $s\in\Z_{\ge 0}$ such that  
\[
c_1(X)\cdot d +g - 1= \sum_{i=1}^{g+1}r_i+2s.
\]
We furthermore assume that the following holds
\begin{equation}\label{equ:parity cc}
\mbox{either}\qquad g=0\qquad \mbox{or}\qquad
r_i= l_{L_i,d}^2+1 \mod 2\qquad \forall i\in\{1,\ldots, g+1\}.
\end{equation}
Choose a configuration $\x$ made of
$r_i$ points on each $L_i$ for $i \in\{1,\ldots, g+1\}$, and
$s$ 
 pairs of $\tau_X$-conjugated 
 points in $X\setminus \R X$.
Given a   $\tau_X$-compatible almost complex structure $J$,
we denote by $\mathcal C(d,\x,L,J)$
 the set of  irreducible real 
$J$-holomorphic curves $f:C\to X$ of genus 
$g$ in $X$, with $f_*[C]=d$, passing through $\x$,
 and such that $f(\R C)\subset L$.
It follows from $(\ref{equ:parity cc})$ that given $f:C\to X\in\mathcal
C(d,\x,L,J)$,  each 
component $L_i$ contains a connected
component of $f(\R C)$, and so $C$ is a maximal real curve.
Furthermore according to Lemma \ref{lem:finite relative}, if
$c_1(X)\cdot d>0$ when $g=1$, then the set $\mathcal C(d,\x,L,J)$ is finite
and composed of nodal immersions for a choice of $J$ that is generic
with respect to all choices made above. 

\begin{defi}\label{defi:absolute W} 
  Let
  $F \in H^{\tau_X}_2(X\setminus L;\Z/2\Z)$. We say that
  \emph{Welschinger invariants} exist for the triple
  $(X_\R,L,F)$ if the integer
$$W_{X_\R,L,F}(d;\underline r,s)= \sum_{C\in\mathcal C(d,\x,L,J)}(-1)^{m_{L,F}(C)} $$
 depends neither on $\x$, $J$, nor on the deformation class of
$X_\R$ as soon as $c_1(X)\cdot d\ne 0$ when $g=1$.
\end{defi}

Note that the notation 
$W_{X_\R,L,F}(d;\underline r,s)$ contains the information about the
 genus $g$ of the curves under enumeration: it is the number of connected
 components of $L$ minus 1. When $L$ is connected (i.e.  $g=0$)  we simply denote
 $W_{X_\R,L,F}(d;s)$ rather than
 $W_{X_\R,L,F}(d;\underline  r,s)$. In this case,
 Welschinger invariants always exist.
\begin{thm}[\cite{Wel1,Wel11,IKS14}]\label{thm:absolute W} 
Let $X_\R$ be a   real compact symplectic manifold of dimension
4 with $\R X\ne \emptyset$. Then  Welschinger invariants exist for \emph{any} triple
 $(X_\R,L,F)$ with  $L$ a connected component of $\R X$  and
 $F \in H^{\tau_X}_2(X\setminus
L;\Z/2\Z)$. 
\end{thm}
 These invariants  were first defined and shown
to exist by Welschinger in \cite{Wel1} for rational curves and 
when $F=[\R X\setminus L]$ (see also the
correction from \cite{Wel11} concerning the appearance of embedded
$J$-holomorphic spheres with self-intersection $-2$ in the proof of
\cite[Theorem 0.1]{Wel1}). Welschinger's seminal work has been
generalizsed
by Itenberg, Kharlamov and Shustin to any $F$ in
\cite{IKS14}, and by Shustin to any $g$ in \cite{Shu14}
for real algebraic del~Pezzo surfaces.
 Thanks to Lemma \ref{lem:finite relative}, it should be possible to adapt in the
obvious way the proof
of \cite{Shu14} in the strategy proposed in \cite{Wel1} (including the
correction from \cite{Wel11}) in order to prove the existence
of Welschinger invariants  for any triple
$(X_\R,L,F)$ (i.e. for curves of higher genus):
 the assumption on $d$  prevent the appearance of non-trivial real ramified
 coverings, while condition $(\ref{equ:parity cc})$ should prevent
the appearance
of real immersions that should be counted with multiplicity two or more,
 along 
a generic path of $\tau_X$-compatible almost complex
structures. Nevertheless, there is a certain amount of technical
details to check that this is indeed the case, which is not the
purpose of this paper.

\begin{rem}\label{rem:two special F}
Among all possible choices of $F$ in $H_2^{\tau_Y}(X\setminus
L;\Z/2\Z)$, the two classes $0$ and $[\R X\setminus L]$ seem to 
play a special role, at least in genus $0$.
In this case, any Welschinger invariant $W_{X_\R,L,F}(d;s)$  either
vanishes or is equal in absolute value to
 $W_{X_\R,L,0}(d;s)$ as soon as $c_1(X)\cdot d-1-2s\ge 2$ and $X_\R$ is
deformation equivalent to a real rational algebraic surface, see
\cite{BP14}.  On the other hand, the invariant $ W_{X_\R,L,[\R
    X\setminus L]}(d;s)$ turns out to be sharp in many situation when 
$c_1(X)\cdot d-1-2s\le 1$, see \cite{Wel4}. Furthermore, we do not
know yet any situation where  $W_{X_\R,L,[\R X\setminus L]}(d;s)$, with
$c_1(X)\cdot d-1-2s\le 1$,
is not
 maximal in absolute value when $F$ ranges over  $H_2^{\tau_Y}(X\setminus
L;\Z/2\Z)$.
\end{rem}

Next theorem, the main result of this paper, 
provides surprisingly very simple relations among Welschinger
invariants of real symplectic $4$-manifolds differing by a special
kind of  surgery. Its proof involves relative Welschinger
invariants defined in next section, and is postponed until Section \ref{sec:proof main}.
\begin{thm}\label{thm:main1}
Let  $X_\R$ be a compact real symplectic manifold of dimension 4, and
let $S$ be a real Lagrangian sphere in $X_\R$, endowed with some
orientation, realizing a 
$\tau_X$-anti-invariant class 
in $H_2(X;\Z)$. We denote by $Y_\R$ the surgery of $X_\R$ along
$S$.

Let also $L$ be 
the union of some
connected components of $\R X$ that is disjoint from
$S$, and 
$F\in H_2^{\tau_Y}(X\setminus L;\Z/2\Z)$ be a class  orthogonal to
$[S]$ such that 
Welschinger invariants exist for both triples $(X_\R,L,F)$ and $(Y_\R,L,F)$.
Then
for 
any  class $d\in H^{-\tau_Y}_2(X;\Z)$, we have
$$ W_{Y_\R,L,F}(d;\underline r,s) = W_{X_\R,L,F}(d;\underline r,s) +
2\sum_{k\ge 1}\
W_{X_\R,L,F}(d-k[S];\underline r,s)$$
whenever $\underline r$ and $s$ are such that
the invariant $ W_{Y_\R,L,F}(d;\underline r,s)$ is defined.
\end{thm}
The fact that $F$ is orthogonal to $[S]$ in $ H_2(X;\Z/2\Z)$
ensures that all invariants $W_{X_\R,L,F}(d-k[S];\underline r,s)$ are also
 defined as soon as $ W_{Y_\R,L,F}(d;\underline r,s)$ is defined (which is
always the case if $g=0$), see Lemma \ref{lem:2 real class}.

\subsection{Relative invariants}\label{sec:relative}
As above, we choose $d\in
H_2^{-\tau_X}(X;\Z)$,
and  $\underline r=(r_1,\cdots,r_{g+1})\in \Z_{\ge 0}^{g+1}$ and $s\in\Z_{\ge 0}$ such that  
$$c_1(X)\cdot d +g - 1 =  \sum_{i=1}^{g+1}r_i+2s.$$
Let also $U=\{E_1,\ldots E_{\kappa}\}$ and $V=\{E'_1,\ldots E'_\lambda\}$ be two
finite (maybe empty) sets of  real embedded
 symplectic spheres in $X_\R$ such that for all
 $i, j\in\{1,\ldots,\kappa\}$ and $u, v\in\{1,\ldots,\lambda\}$, we have:
\begin{itemize}
\item $[E_i]^2=[E_u']^2=-2$,
\item $ E_i\cap E_j=E'_u\cap E'_v=E_i\cap E'_u=\emptyset$ if $i\neq j$
  and $u\neq v$,
\item $\R E'_i\ne\emptyset$,
\item $d\cdot[E'_u]=0$,
\item $\sum_{u=1}^\lambda [\R E_u']=0\in H_1(\R X;\Z/2\Z)$.
\end{itemize}

Let $L_1,\ldots, L_{g+1}$ be the topological closure of $g+1$ connected components of 
$\displaystyle \R X\setminus\bigcup_{u=1}^\lambda\R E'_u$ such that
$$L_i\cap L_j=\emptyset \quad\mbox{if }i\ne j \qquad\mbox{and}\qquad\left(\bigcup_{i=1}^{g+1}L_i\right)\cap
\left(\bigcup_{i=1}^{\kappa}\R E_i\right)=\emptyset. $$
Denote $\displaystyle L=\bigcup_{i=1}^{g+1}L_i$.
Let also $L_0$ be the topological closure of the union of some connected
components of $\displaystyle \R X\setminus\bigcup_{u=1}^\lambda\R E'_u$ such that
$$L_0\cap L=\emptyset\qquad \mbox{and}
\qquad \partial (L_0 \cup  L) = \bigcup_{u=1}^\lambda\R E_u'.$$
In particular, the last condition implies that each circle $\R E'_v$
is contained in the
boundary of a connected component of $L_0\cup L$ and of a connected component of  
$\R X\setminus\left(L_0\cup L\right).$
Given $C$  a real symplectic curve and 
$f:C\to X$ a real symplectic immersion, we denote by $l_{L_i, f_*[C]}$
the image of $l_{\R X, f_*[C]}$ by the natural map
$H_1(\R X;\Z/2\Z)\to H_1(L_i,\partial L_i;\Z/2\Z) $.
We still assume that $(\ref{equ:parity cc})$ holds.

Choose a configuration $\x$ made of $r_i$ points on each $L_i$ for
$i\in\{1,\cdots,g+1\}$,  and  $s$ 
 pairs of $\tau_X$-conjugated 
 points in $X\setminus \R X$.
Given a   $\tau_X$-compatible almost complex structure $J$ such that all
symplectic curves in $U\cup V$ are $J$-holomorphic,
we denote by $\mathcal C(d,\x,L,U,V,J)$
 the set of irreducible real 
$J$-holomorphic curves $f:C\to X$  in $X$ of genus 
$g$, with $f_*[C]=d$, whose image is not contained in $U\cup
 V$,
 passing through $\x$,
 and such that $f(\R C)\subset L$.
Given $f:C\to X\in\mathcal
C(d,\x,L,U,V,J)$,  condition $(\ref{equ:parity cc})$ forces each 
component $L_i$ to contain a connected
component of $f(\R C)$, and so $C$ is a maximal real curve. Furthermore according to Lemma \ref{lem:finite relative}, if
$c_1(X)\cdot d>0$ in the case $g=1$, then the set $\mathcal C(d,\x,L,U,V,J)$ is finite
and composed of simple nodal immersions for 
 a choice of $J$ that is generic
with respect to all choices made above.

Let  $f:C\to X$ be an element of $\mathcal C(d,\x,L,U,V,J)$, and
 choose $C^+$ to be the topological closure of one of the halves of $C\setminus\R C$.
As in the case of absolute invariants, 
given $F\in H^{\tau_X}_2(X\setminus L;\Z/2\Z)$
 we
define the
 $(L\cup L_0,F)$-mass of $f$ as
$$m_{L\cup L_0,F}(f)=m(f)+[f(C^+)]\cdot F, $$
where $m(f)$ is the number of real elliptic 
nodes of  $f(C)$ in $L\cup L_0$. Again, $m_{L\cup L_0,F}(f)$ does not depend on
the chosen  half of $C\setminus\R C$.

Let $\overline X_\R$ be the successive real surgeries of $X_\R$
along the curves $E'_1,\ldots,E'_\lambda$, and 
$\overline L$ and $\overline L_0$  be the union of the connected components  of
$\R \overline X$ obtained by gluing disks respectively to the boundary of $L$
and $L_0$. 
Next theorem is a consequence of Corollary \ref{cor:reduction} and
Theorem \ref{thm:relations 2} that we prove in next sections.
\begin{thm}\label{thm:relative}
Let $F\in H^{\tau_X}_2(X\setminus L;\Z/2\Z)$
 orthogonal in $H_2(X;\Z/2\Z)$ 
to all classes realized by
the curves in $V$. If Welschinger invariants exist for
the triple $(\overline X_\R,\overline L,F+[\overline L_0])$,
then the
integer  
$$W^{U,V}_{X_\R,L,L_0,F}(d;\underline r,s)=\sum_{C\in\mathcal C(d,\x,L,U,V,J)}(-1)^{m_{L\cup L_0,F}(C)} $$
depends neither on $\x$, $J$, nor on the deformation class of
the 5-tuple $(X_\R,L,L_0,U,V)$.
\end{thm}

When the conclusion of Theorem \ref{thm:relative} holds, we call
the numbers $W^{U,V}_{X_\R,L,L_0,F}$  \emph{Welschinger invariants of $X_\R$
   relative to the pair $(U,V)$}.

\subsection{Surgery and enumerative geometry}\label{sec:surgery
  and enumerative}
We keep using notations and choices introduced in Section
\ref{sec:relative}. 

\subsubsection{}
Suppose  that $U\ne\emptyset$, and let $E\in U$. We denote by
$\widehat U=U\setminus{E}$, and by $Y_\R$ the surgery of $X_\R$
along $E$. 
If $\R E\ne \emptyset$, recall that $\R Y$ is obtained topologically by cutting $\R X$
along $\R E$ and gluing back a disk along each boundary circle. 
If $\R E \cap L_0\ne\emptyset$, then we denote by $\widehat L_0$ the
union of $L_0$ with the two glued disks. Otherwise we set $\widehat L_0=L_0$.
Next theorem is proved in Section \ref{sec:relations}.

\begin{thm}\label{thm:relations 1}
We have
\begin{equation}\label{equ:relation 1}
 W^{\widehat U,V}_{X_\R,L,L_0,F}(d;\underline r,s) = \sum_{k\ge 0}\binom{\frac{1}2 d\cdot [E]
  +2k}{k} \ W^{U,V}_{X_\R,L,L_0,F}(d-2k[E];\underline r,s).
\end{equation}
If furthermore $d\cdot [E]=0$, and $F\in[E]^\perp$ in $H_2(X;\Z/2\Z)$, then
$$ W^{\widehat U,V}_{Y_\R,L,\widehat L_0,F}(d;\underline r,s) = \sum_{k\ge 0}2^k \ W^{U,V}_{X_\R,L,L_0,F}(d-k[E];\underline r,s).$$
\end{thm}

According to Lemma \ref{lem:2 real class}, we have
$F\in H_2^{\tau_Y}(X\setminus L;\Z/2\Z)$ if
  $F\in[E]^\perp$ in $H_2(X;\Z/2\Z)$.
In particular  the number $ W^{\widehat
  U,V}_{Y_\R,L,\widehat L_0,F}(d;\underline r,s) $ in Theorem \ref{thm:relations 1}
is well defined.

In next corollary, we use the convention that
$$\binom{-1}{0}=0.$$
\begin{cor}\label{cor:reduction}
The number
$W^{U,V}_{X_\R,L,L_0,F}(d;\underline r,s)$ can be expressed in terms of the numbers
$W^{\widehat U,V}_{X_\R,L,L_0,F}(d-2k[E];\underline r,s)$ with $k\ge 0$. More precisely,
we have
$$W^{U,V}_{X_\R,L,L_0,F}(d;\underline r,s)=\sum_{k\ge 0}(-1)^{k}\left(
\binom{\frac{1}2 d\cdot [E] +k}{\frac{1}2 d\cdot [E]} + 
\binom{\frac{1}2 d\cdot [E] +k-1}{\frac{1}2 d\cdot [E]} 
\right) W^{\widehat U,V}_{X_\R,L,L_0,F}(d-2k[E];\underline r,s).$$
\end{cor}
\begin{proof}
Relation $(\ref{equ:relation 1})$ expresses the numbers
$W^{\widehat U,V}_{X_\R,L,L_0,F}(d-2k[E];\underline r,s)$ in terms of the numbers
$W^{U,V}_{X_\R,L,L_0,F}(d-2(k+l)[E];\underline r,s)$, $k,l\ge 0$. Since this is an
upper triangular linear system 
with coefficients 1 on the diagonal, it is invertible. 
The exact expression of $W^{U,V}_{X_\R,L,L_0,F}(d;\underline r,s)$ in terms of the
numbers $W^{\widehat U,V}_{X_\R,L,L_0,F}(d-2k[E];\underline r,s)$ follows from
Proposition \ref{prop:inverse matrix}.
\end{proof}

\begin{prop}\label{prop:inverse matrix}
Let $m\in \Z_{\ge 0}$ and $K\in \Z_{\ge 1}$, and let $M$ be the matrix
$$M=\left(
\binom{m +2(j-1)}{j-i} 
\right)_{1\le i,j\le K}. $$
Then the matrix $M$ is invertible and we have
$$M^{-1}=\left((-1)^{i+j}\left(
\binom{m +i+j-2}{m+2i-2} + 
\binom{m +i+j-3}{m+2i-2} 
\right)\right)_{1\le i,j\le K}.$$
\end{prop}
\begin{proof}
Let $(\nu_{i,j}))_{1\le i,j\le K}$ be 
 the product of the two above matrices. Then we have
\begin{align*}
\nu_{i,j}&=\sum_{k=1}^K (-1)^{k+j} \binom{m +2(k-1)}{k-i} 
\left(
\binom{m +k+j-2}{m+2k-2} + 
\binom{m +k+j-3}{m+2k-2} 
\right)
\\ &=  (-1)^{j}\ \sum_{k=i}^j (-1)^{k} \binom{m +2(k-1)}{k-i} 
\left(
\binom{m +k+j-2}{m+2k-2} + 
\binom{m +k+j-3}{m+2k-2} 
\right).
\end{align*}
Defining $m'=m+2i-2$ and $I=j-i$, we get
\begin{align*}
\\ \nu_{i,j}&=  (-1)^{I}\ \sum_{k=0}^{I} (-1)^{k} \binom{m'+2k}{k} 
\left(
\binom{m'+I +k}{m'+2k} + 
\binom{m'+I +k-1}{m'+2k} 
\right).
\end{align*}
Using the identity
$$\binom{u}{v}\binom{v}w=\binom u w \binom{u-w}{v-w}$$
we obtain
\begin{align*}
\\ \nu_{i,j}&=  (-1)^{I}\ \sum_{k=0}^{I} (-1)^{k} \left(
\binom{m'+I+k}{k} \binom{m'+I}{m'+k} +
\binom{m'+I +k-1}{k}\binom{m'+I -1}{m'+k} 
\right).
\end{align*}
In the case when $I=0$, this gives $\nu_{i,i}=1$. In the case when
$I>0$ we have
\begin{align*}
\\ \nu_{i,j}&=  (-1)^{I}\ \sum_{k=0}^{I} (-1)^{k} \left(
\binom{m'+I+k}{m'+I} \binom{m'+I}{m'+k} +
\binom{m'+I +k-1}{m'+I-1}\binom{m'+I -1}{m'+k} 
\right).
\end{align*}
It follows from \cite[Identity 5.25]{GKP94} that
\begin{align*}
\sum_{k=0}^{I} (-1)^{k}
\binom{m'+I+k}{m'+I} \binom{m'+I}{m'+k} &=(-1)^I
\\ &= -\sum_{k=0}^{I} (-1)^{k}\binom{m'+I +k-1}{m'+I-1}\binom{m'+I -1}{m'+k}.
\end{align*}
Hence $\nu_{i,j}=0$ if $i\ne j$, and the proposition is proved.
\end{proof}

\subsubsection{}
Suppose  that $V\ne\emptyset$, and let $E\in V$. We denote by
$\widehat V=V\setminus{E}$, and by $Y_\R$ the surgery of $X_\R$
along $E$. 
Let $L_i$ be the connected component of $L\cup L_0$ whose boundary contains
$\R E$. 
Recall that $\R Y$ is obtained topologically by cutting $\R X$
along $\R E$ and gluing back a disk along each boundary circle. 
We
define $\widehat L_j=L_j$ for $j\ne i$, and by $\widehat L_i$ the
union of $L_i$ with the corresponding glued disk. Accordingly, we
define
$$\widehat L=\bigcup_{j=1}^{g+1}\widehat L_j .$$
Next theorem is proved in Section \ref{sec:relations}.

\begin{thm}\label{thm:relations 2}
Let $F\in H^{\tau_X}_2(X\setminus L;\Z/2\Z)$
 orthogonal in $H_2(X;\Z/2\Z)$ 
to all classes realized by
the curves in $V$. Then we have
$$ W^{U,\widehat V}_{Y_\R,\widehat L,\widehat L_0,F}(d;\underline r,s) =  W^{U,V}_{X_\R,L,L_0,F}(d;\underline r,s).$$
\end{thm}
The assumption of  Theorem \ref{thm:relations 2} ensures that 
 the number $ W^{U,\widehat V}_{Y_\R,\widehat L,\widehat L_0,F}(d;\underline r,s) $ 
is well defined by Lemma \ref{lem:2 real class}.

\subsection{Proof of Theorems \ref{thm:main1}, \ref{thm:relative}, \ref{thm:relations 1}, and \ref{thm:relations 2}}\label{sec:relations}
\subsubsection{Symplectic sums}\label{sec:sympl sum}
Here we describe a
very particular case of
 the symplectic sum 
formula from
\cite{IP,TehZin14}.
Since we are working in this paper only with symplectic sums of
$4$-dimensional manifolds (which are in particular 
semi-positive) along spheres (for which the so-called $S$-matrix is
the identity),
none of the major issues with \cite{IP} raised in
  \cite{TehZin14} are relevant in our situation.
Let $(Z,\omega_Z)$ be a compact and connected symplectic
manifold of dimension 4 containing an embedded symplectic sphere $E$
with $[E]^2=-2$. We furthermore assume the existence of a
symplectomorphism $\phi$ from $E$ to 
 a symplectic curve
realizing the class $l_1+l_2$ in 
$(\C P^1\times \C P^1,\omega_{FS}\oplus \omega_{FS})$.
 By abuse, we still denote by $E$ the image $\phi(E)$ in $\C P^1\times \C P^1$.
Since the self-intersection of $E$ in $\C P^1\times \C P^1$ and $Z$ are opposite, 
there exists a symplectic bundle isomorphism $\psi$
between the  normal bundle of $E$ in $\C P^1\times \C P^1$ and  
the dual of the normal bundle of $E$ in $Z$.
Out of these data, one produces a family of symplectic
$4$-manifolds $(\mathcal Z_t,\omega_{t})$ parametrized by a small complex
number $t$ in $\C^*$, see \cite{Gom95}. All these manifolds are deformation
equivalent, and are called \emph{symplectic sums of $(\C P^1\times \C P^1,\omega_{FS}\oplus \omega_{FS})$ and
$(Z,\omega_Z)$ along $E$}.
This family can be seen as a symplectic
deformation of the singular 
symplectic manifold $X_\sharp= Z\cup_E \left(\C P^1\times \C P^1\right)$ 
obtained by gluing  $(Z,\omega_Z)$ and $(\C P^1\times \C P^1,\omega_{FS}\oplus \omega_{FS})$ along
$E$.

\begin{prop}[{\cite[Theorem 2.1]{IP}},{\cite[Proposition 3.1]{TehZin14}}]\label{thm:degen}
There exists a symplectic $6$-manifold $(\mathcal Z, \omega_{\mathcal Z})$
and a symplectic fibration $\pi : \mathcal Z\to D$ over a disk $D\subset\C$ such that
the central fiber $\pi^{-1}(0)$ is the singular symplectic
manifold $X_\sharp$, and $\pi^{-1}(t)=(\mathcal Z_{t},\omega_{t})$
for $t\ne 0$.
\end{prop}
Topologically, $\mathcal Z_t$ is simply the connected sum along $E$ of $Z$ with
$\C P^1\times \C P^1$. Since $\C P^1\times \C P^1\setminus E$ is 
the normal bundle of $E$ in $Z$,  this connected sum is trivial
and we have $\mathcal Z_t=Z$.
Furthermore, without loss of generality we may assume that the homology
class
realized by $E$ in its normal bundle is the restriction of the class
$l_1-l_2$ in $H_2(\C P^1\times \C P^1;\Z)$.

\medskip
In addition to $E$, suppose that $Z$ contains a collection 
$W=\{E_1,\ldots,E_\kappa\}$ of pairwise disjoint
embedded symplectic spheres
with $[E_i]^2=-2$ that are all disjoint from $E$.
Let $d \in H_2(\mathcal Z_{t};\Z)$, and 
choose  $\x(t)$ 
a  set of $c_1(X)\cdot d+g-1$ 
symplectic sections 
$$ D \to  \mathcal Z\setminus \bigcup_{i=1}^\kappa E_i$$
 such that 
$\x(0)\cap \left(E\displaystyle \bigcup_{i=1}^\kappa E_i\right)=\emptyset$. Choose an almost complex structure
$J$ on $\mathcal Z$ tamed by $\omega_{\mathcal Z}$, which restricts to 
an almost complex structure
$J_t$ tamed by $\omega_t$ on each fiber 
$\mathcal Z_t$ and for which all curves $E_1,\ldots,E_\kappa$ are
$J_t$-holomorphic. We assume that $J$ is 
 generic with respect to all choices
we made. 
Recall that the set $\CC^\C(d,g,\x(t),W,J_t)$ for $t\ne 0$ has been
defined in Section \ref{sec:preliminaries}.
We  define $\CC^\C(d,g,\x(0),W,J_0)$ to be
the set $\left\{\overline f:\overline C  \to X_\sharp \right\}$
of limits, as stable maps, 
 of
  maps in $\CC^\C(d,g,\x(W,t),J_t)$ as $t$
  goes to $0$.
Recall (see {\cite[Section 3]{IP}}, or comments following
{\cite[Theorem 1.1]{TehZin14}}) that $\overline C$ is a connected nodal
 curve of arithmetic genus $g$  such that:
\begin{itemize}
\item $\x(0)\subset \overline f(\overline C)$;
\item any point $p\in \overline f^{\ -1}(E)$ is a node of 
$\overline C$ which is the intersection of two 
irreducible components $\overline C'$ and $\overline C''$ of
$\overline C$, with  $\overline f(\overline
C')\subset Z$ and $\overline f(\overline
C'')\subset  \CP^1\times\CP^1$;

\item if in addition neither $\overline f(\overline
C')$ nor $\overline f(\overline
C'')$ is entirely mapped to $E$, then
the multiplicity of intersection
of both 
$\overline f(\overline
C')$ and $\overline f(\overline
C'')$ with $E$ are equal.
\end{itemize}
Formally, {\cite[Section 3]{IP}} only deals with the case $\kappa=0$, 
  however the value of  $\kappa$ plays no role there.
Given  an element $\overline f:\overline C\to X_\sharp$ 
of $\CC^\C(d,g,\x(0),W,J_0)$, we denote by $C_1$ (resp. $C_0$)
 the union of the irreducible
components of $\overline C$ mapped to $Z$ (resp. $\CP^1\times\CP^1$).
The next three statements are proved in
 \cite[Section 3.2]{BP14}. Again formally,
\cite[Lemma 3.6, Propositions 3.7, and Corollary 3.8]{BP14} are stated for $g=0$ and $\kappa=0$, however neither
$g$ nor $\kappa$  plays no role in their proof once we replace
\cite[Proposition 3.2]{BP14} by Lemma \ref{lem:finite relative} above.
\begin{lemma}\label{lem:homology}
Given an element $\overline f:\overline C\to X_\sharp$
of $\CC^\C(d,g,\x(0),W,J_0)$,
there exists
$k\in\Z_{\ge 0}$ such that
$$
\overline f_*[C_1] = d -  k[E] \quad \textrm{and} \quad
\overline f_*[C_0]= kl_1 + (d\cdot [E] +k)l_2 .
$$
Moreover $c_1(Z)\cdot \overline f_*[C_1]=c_1(\mathcal Z_t)\cdot d$.
\end{lemma}

\begin{prop}\label{prop:degeneration}
Assume that $\x(0)\subset Z $.
Then for a generic   $J_0$,  the set
$\CC^\C(d,g,\x(0),W,J_0)$ is finite, and only depends on $\x(0)$ and
$J_{0}$.
 Given
$\overline f:\overline C\to  X_\sharp$
  an element of  $\CC^\C(d,g,\x(0),W,J_0)$, the
restriction of $\overline f$ to any component of $\overline C$ is a
simple map, and
no irreducible component of $\overline C$ is
  entirely mapped to $E$. Moreover 
 the curve $C_1$ is irreducible, 
and the image of any irreducible component of $C_0$ realizes
a class $l_i$.
The map $\overline f$ is the limit of a unique element of
$\CC^\C(d,g,\x(t),W,J_t)$ as $t$ goes to 0.
\end{prop}

Next Corollary generalizes \cite[Corollary 3.8]{BP14} and Abramovich-Bertram-Vakil formula
\cite[Theorem 3.1.1]{AB}
and \cite[Theorem 4.2]{Vak2}.
\begin{cor}\label{lem:lem2}
Suppose that $\x(0)\subset Z$, and let $\overline f:\overline C\to  X_\sharp$
be  an element of  $\CC^\C(d,g,\x(0),W,J_0)$.
Define 
 $\CC_{\overline f}$ to be the set of elements  $\overline
 f':\overline C'\to X_\sharp$ in $\CC^\C(d,g,\x(0),W,J_0)$ such that 
 $\overline f_{| C_1}=\overline f'_{| C'_1}$.
If $\overline f_*[ C_1]=d-k[E]$, then $\CC_{\overline f}$ contains
exactly  $\binom{d\cdot 
  [E]+2k}{k}$ elements.
\end{cor}

\subsubsection{Proof of Theorems \ref{thm:relations 1} and \ref{thm:relations 2}}
We proved these theorems
using the obvious equivariant version of Section \ref{sec:sympl sum},
 which exists by 
\cite[Lemma 2.14]{Teh10} and Remark \ref{rem:real g}.
In the case when $g=0$, the curve $E$ is the only element of $U$, and
$V$ is empty, Theorem 
\ref{thm:relations 1} is exactly \cite[Theorem 2.5(1)]{BP14}.
One simply has to note that since $\R E\cap L=\emptyset$, all
intersection points of $\overline f(\overline C_1)$ and $E$ are
$\tau_X$-conjugated for $\overline f\in \CC^\C(d,0,\x(0),W,J_0)$.
In particular we must have
$$
\overline f_*[C_1] = d -  2k[E] \quad \textrm{and} \quad
\overline f_*[C_0]= 2kl_1 + (d\cdot [E] +2k)l_2 .
$$
with $k\in\Z_{\ge 0}$.

In the case when $g=0$, the curve $E$ is the only element of $V$, and $U$ is empty, Theorem
\ref{thm:relations 2} is exactly \cite[Theorem 2.5(2)]{BP14}.

In general, 
Theorems \ref{thm:relations 1} and \ref{thm:relations 2} are obtained by the
straightforward adaptation of the proof
of \cite[Theorem 2.5]{BP14} replacing  \cite[Section 3.2]{BP14} by
Section \ref{sec:sympl sum}. 
\hfill \smiley

\subsubsection{Proof of Theorem \ref{thm:relative}}
By induction, it follows from Corollary \ref{cor:reduction} and
Theorem \ref{thm:relations 2} that the numbers
$W^{U,V}_{X_\R,L,L_0,F}(d;\underline r,s)$ can be expressed in terms of 
the numbers $W^{\emptyset,\emptyset}_{\overline X_\R,\overline L,\overline
  L_0,F}(d;\underline r,s)$. By definition we have
$$W^{\emptyset,\emptyset}_{\overline X_\R,\overline L,\overline
  L_0,F}(d;\underline r,s)= W_{\overline X_\R,\overline L,F+[\overline L_0]}(d;\underline r,s),$$
so Theorem \ref{thm:relative} now follows from
our hypothesis.
\hfill\smiley

\subsubsection{Proof of Theorem \ref{thm:main1}}\label{sec:proof main}
According to Lemma \ref{lem:sphere invariant}, we have $[S]\cdot d=0$,
hence Theorem \ref{thm:relations
 1} and Corollary \ref{cor:reduction} give
$$ W_{Y_\R,L,F}(d;\underline r,s) = \sum_{k,l\ge 0} (-1)^l \ 2^k
\ \left(
\binom{k+l}{k} + 
\binom{k+l-1}{k} 
\right)\ W_{X_\R,L,F}(d-(k+2l)[S];\underline r,s).$$
The total coefficient of the invariant  $W_{X_\R,L,F}(d-i[S];\underline r,s)$
appearing in this sum is equal to $1$ if $i=0$, and is equal to 
\begin{align*}  \sum_{k+2l=i} (-1)^l \ 2^k
\ \left(
\binom{k+l}{k} + 
\binom{k+l-1}{k} 
\right)=& \sum_{k+2l=i} (-1)^l \ 2^k
\ 
\binom{k+l}{k} - \sum_{k+2l=i-2} (-1)^l \ 2^k
\ 
\binom{k+l}{k}
\\ =& \ u_{i} -u_{i-2}
\end{align*}
otherwise, where
$$u_i=\sum_{k+2l=i} (-1)^l \ 2^k
\ 
\binom{k+l}{k}.$$
 By  Pascal's rule, we have
$$ u_{i+2}=2u_{i+1}-u_{i},$$
hence we deduce from the values $u_0=1$ and $u_1=2$ that
$$u_i=i+1\qquad \forall i\ge 0, $$
and the theorem is proved.
\hfill\smiley

\section{Applications to real algebraic rational surfaces}\label{sec:rational surf}

From now on, we restrict ourselves to the study of real symplectic
manifolds that are deformation equivalent to a
real algebraic rational surface.
We refer for
example to \cite{Sil89,Kol97,DK,DegKha02} for an account on real
algebraic rational surfaces.
 Their classification up to deformation has been established in
\cite{DegKha02}. In particular, any two real algebraic $\R$-minimal rational
surfaces with a non-empty real part are deformation equivalent if and
only if their are deformation equivalent as complex algebraic surfaces
and if their real part are homeomorphic.
Furthermore, it follows from  this classification and Example \ref{ex:gluing real part}
that the  surgery of a real algebraic rational
surfaces along a real Lagrangian sphere contained in $\R X$ remains
deformation equivalent to a real algebraic rational surface.
Recall also that it follows from 
\cite[Corollary 2.6 and Section 4]{BP14} that any genus 0 Welschinger
invariant $W_{X_\R,L,F}(d;s)$ of a real rational algebraic surface $X_\R$ is equal, up to
a well defined sign, to 
$W_{X_\R,L,[R]}(d;s)$, where
$R\subset \R X\setminus L$ only depends on $L$ and $F$.

\medskip
By the blow-up of a real algebraic rational surface $X_\R$, we mean
the blow-up of $X_\R$ at a finite number of real points and pairs of
complex conjugated points, equipped with the real structure that turns
the blowing down map into a real map.

\begin{prop}\label{cor:spheres}
Let $X_\R$ be a real algebraic rational surface with a non-empty real part, and $L$ be a 
connected component of $\R X$  such that $\R X\setminus L$ is a
disjoint union of spheres. Then by  finitely many successive applications of Theorem
\ref{thm:main1}, all genus 0 Welschinger
invariants $W_{X_\R,L,F}$ can be
computed out of genus 0
Welschinger
invariants of either a blow-up of the real projective plane or a
blow-up 
of a real
quadric in $\CP^3$.
\end{prop}
\begin{proof}
As mentioned above, we may suppose that $F=[R]$ where  $R$
is
the union of some
connected components of $\R X\setminus L$. 
Since a
connected component of $\R X\setminus L$ is a real Lagrangian sphere, it is
orthogonal to $[R]$ in $H_2(X\setminus L;\Z/2\Z)$. Hence by 
applying Theorem \ref{thm:main1}  to $W_{X_\R,L,[R]}(d;s)$ and by
 successive surgeries along all spheres in  $\R X\setminus L$,
 we are reduced to the case when $\R X$ is connected.
It follows from \cite[Main Theorem]{DegKha02} that a real
algebraic rational
surface with a connected real part is  deformation equivalent 
to either a blown-up real projective plane or a  blown-up real
quadric in $\CP^3$, so the statement is proved.
\end{proof}

Genus 0 Welschinger invariants of  any
blow-up of the real projective plane, and of any blow-up of a real
quadric in $\CP^3$ are computed in \cite{HorSol12}. Hence combining
Proposition \ref{cor:spheres} and \cite{HorSol12}, one can compute any
genus 0 Welschinger
invariant $W_{X_\R,L,F}$ as soon as $\R X\setminus L$ is a disjoint
union of spheres. For example, genus 0 Welschinger invariants of
$\R$-minimal real conic bundles can be deduced from Welschinger invariants
of a quadric ellipsoid in $\CP^3$ blown-up at several  pairs of complex
conjugated points.

Analogously, one can reduce the computation of genus 0 Welschinger invariants
of del~Pezzo surfaces to the case when the real part of the surface
has at most two connected components. These latter cases have been
covered in \cite{Bru14}. 
\begin{prop}\label{cor:dp1 2 comp}
Let $X_\R$ a real algebraic del~Pezzo surface, and $L$ be a 
connected component of $\R X$.
 Then by  finitely many successive applications of Theorem
\ref{thm:main1},  all genus 0 Welschinger
invariants $W_{X_\R,L,F}$ can be computed out of genus 0  Welschinger
invariants of a real algebraic del~Pezzo surface of the same degree
with a real part consisting of at most two connected components.
\end{prop}
\begin{proof}
It follows from the classification of real algebraic del~Pezzo surfaces
that if $\R X$ has three or more connected components, then 
at  most one of them is not
homeomorphic to a sphere. So the proof is analogous to the proof
of Proposition \ref{cor:spheres}.
\end{proof}

We can  generalize Propositions \ref{cor:spheres} and
\ref{cor:dp1 2 comp} to relative Welschinger invariants as follows. 

\begin{thm}\label{thm:cor 2 cc}
Let $X_\R$ be a real algebraic rational surface  with a
disconnected real part, and let $L$ be a
 connected component of $\R X$. 
 Then  by  finitely many successive applications of Theorems
\ref{thm:main1} and \ref{thm:relations 2}, all absolute genus 0 Welschinger
invariants $W_{X_\R,L,F}$ can be computed out of  relative genus 0 Welschinger
invariants of a blow-up $Y_\R$ of the 
$\R$-minimal real conic bundle with $2$ spheres as real components.

Furthermore by  finitely many successive applications of Theorem
\ref{thm:main1}, all absolute genus 0 Welschinger
invariants $W_{X_\R,L,0}$ and  $W_{X_\R,L,[\R X\setminus L]}$
 can be respectively computed out of the absolute genus 0 Welschinger
invariants 
$W_{Y_\R,L,0}$ and  $W_{Y_\R,L,[\R Y\setminus L]}$.
\end{thm}
\begin{proof}
The case when $\R X$ has two connected components  follows
from the classification of real algebraic rational surfaces up to
deformations \cite[Main Theorem]{DegKha02}. Hence we suppose now that
$\R X$ has a least three connected components.
  As above, 
it is enough to consider the case when 
$F$ is realized by
the union $R$ of some connected components of
$\R X\setminus L$. 
It follows again
from 
\cite[Main Theorem]{DegKha02} that
$\R X$ can be degenerated to a real algebraic  rational nodal
surface $Z_\R$ such that $L$ is contained in the non-singular locus
of $Z_\R$, and such that $\R Z\setminus L$ is connected.
The non-singular real algebraic rational surface $Y_\R$ obtained by
blowing up all nodes of $Z_\R$ has a real part consisting of exactly
two connected components.
Hence
it follows from the
classification of real algebraic rational surfaces that $Y_\R$ is, up to
deformation,  
a blow-up of the 
$\R$-minimal real conic bundle with $2$ spheres as real components.
Thanks to Theorems \ref{thm:main1} and  \ref{thm:relations 2} and Example \ref{ex:gluing real part},
 all genus  0 Welschinger invariants $W_{X_\R,L,[R]}$ can be expressed
in terms of genus 0 Welschinger invariants of $Y_\R$ relative to the vanishing
cycles of the degeneration of $X_\R$ to $Z_\R$, so the first statement is proved.
The statement about $W_{X_\R,L,0}$ and  $W_{X_\R,L,[\R X\setminus
    L]}$ now follows from Theorem \ref{thm:main1b}.
\end{proof}
\begin{rem}
Theorem \ref{thm:cor 2 cc} and its proof 
can of course be generalized to Welschinger
invariants of higher genus. We decided nevertheless 
to restrict to the genus 0 case
since it makes the statement much shorter.
\end{rem}
\begin{cor}
Let $X_\R$ be a real algebraic rational surface  with a
disconnected real part,  let $L$ be a
 connected component of $\R X$, and let $Z_\R$ be the real algebraic
 rational surface  obtained by blowing up $X_\R$ at some point
 $p\in\R X\setminus L$. 
Then the genus 0 Welschinger invariants  $W_{Z_\R,L,0}$ and  $W_{Z_\R,L,[\R
    Z\setminus L]}$ do not depend on the connected component of 
$\R X\setminus L$ containing $p$.
\end{cor}

\section{A few concrete computations}\label{sec:concrete}
Here we illustrate  Theorem \ref{thm:main1} 
 with some explicit computations of absolute genus 0 Welschinger invariants.
Again, we refer for
example to \cite{Sil89,Kol97,DK,DegKha02}) for
 the classification of real algebraic rational surfaces.

\subsection{Real cubic surface with two real components}
Let $Y_\R$ be a non-singular real algebraic cubic surface in $\CP^3$
whose real part is homeomorphic to the disjoint union of a sphere $S$ 
and a real projective plane $\RP^2$, and let $X_\R$ be  the  surgery of
$Y_\R$ along $S$. 
As a first and  easy  example of application of Theorem \ref{thm:main1},
we show how to compute absolute Welschinger invariants of $Y_\R$ out
of those to $X_\R$.

By Example \ref{ex:cubic surface} and 
the classification of real cubic surfaces, we may assume that  $X_\R$
is  $\CP^2$ blown up in three pairs of complex conjugated
points.
Let us denote by $E_1,\ldots,E_6$ the six corresponding exceptional
divisors, and let $D$ be the pull back of a line in $\CP^2$ not
passing through the six blown up points.
Without loss of generality, we may assume that
the following identities hold in $H_2(X;\Z)$:
$$c_1(X)=3[D]-\sum_{i=1}^6[E_i] \qquad\mbox{and}\qquad
[S]=2[D]-\sum_{i=1}^6[E_i] . $$
By Theorem  \ref{thm:main1} we have
$$W_{Y_{\R},\RP^2,F}(c_1(X);s)=W_{X_{\R},\RP^2,F}(c_1(X);s)
+2 W_{X_{\R},\RP^2,F}([D];s).$$
It is well known that
\begin{align*}
&W_{X_{\R},\RP^2,0}(c_1(X);s)=W_{\CP^2,\RP^2,0}(3[D];s+3)=2-2s,
\\&W_{X_{\R},\RP^2,0}([D];s)=W_{\CP^2,\RP^2,0}([D];s)=1.
\end{align*}
Combining with the identity
\begin{equation}\label{equ:sign W}
W_{X_{\R},\RP^2,[S]}(d;s)=(-1)^{\frac{d\cdot
    [S]}2}W_{X_{\R},\RP^2,0}(d;s),
\end{equation}
we obtain
$$W_{Y_{\R},\RP^2,0}(c_1(X);s)=4-2s\qquad
\mbox{and}\qquad
W_{Y_{\R},\RP^2,[S]}(c_1(X);s)=-2s.$$

\begin{rem}\label{rem:rem1}
By \cite[Theorems 1.2(1) and 1.2]{BP14},
the vanishing of $W_{Y_{\R},\RP^2,[S]}(c_1(X);0)$ is actually a
general fact. Hence we could also have used this general vanishing
result to deduce the value of $W_{X_{\R},\RP^2,0}(c_1(X);0)$ out of the
 value 
$W_{\CP^2,\RP^2,0}([D];s)=1$ and the equation
$$0=W_{X_{\R},\RP^2,0}(c_1(X);0)
-2 W_{\CP^2,\RP^2,0}([D];0) $$
given by the combination of  $(\ref{equ:sign W})$, Theorem
\ref{thm:main1}, and \cite[Theorems 1.1(1) and 1.2]{BP14}.
\end{rem}

By Theorem  \ref{thm:relations 1} we have
$$W_{Y_{\R},\RP^2,F}(2c_1(X);s)=W_{X_{\R},\RP^2,F}(2c_1(X);s)
+2W_{X_\R,\RP^2,F}(4[D]-\sum_{i=1}^6[E_i];s) +2 W_{X_\R,\RP^2,F}(2[D];s). $$
The following values are taken from \cite{IKS2,Br8} and \cite[Table
  4]{Bru14}.
\begin{center}
\begin{tabular}{ |c|c| c| c|}
\hline $s  $  & $W_{X_{\R},\RP^2,0}(2c_1(X);s)$&
  $W_{X_{\R},\RP^2,0}(4[D]-\sum_{i=1}^6[E_i];s)$ & $W_{X_{\R},\RP^2,0}(2[D];s)$
\\ & & & 
\\\hline
  $0$   &  78 & 40 & 1
\\\hline 
$1$ &  30  &  16& 1
\\\hline
$2$ &  22 & 0 & 1
\\\hline
\end{tabular}
\end{center}
Combining this with $(\ref{equ:sign W})$ we obtain
\begin{center}
\begin{tabular}{ |c|c| c| }
\hline $s  $  & $W_{Y_{\R},\RP^2,0}(2c_1(X);s)$ & $W_{Y_{\R},\RP^2,[S]}(2c_1(X);s)$
\\ & & 
\\\hline
  $0$   &  160 & 0
\\\hline 
$1$ &  64 & 0
\\\hline
$2$ &  24 & 24
\\\hline
\end{tabular}
\end{center}
in accordance with \cite[Table 4]{Bru14}.

\begin{rem}\label{rem:rem2}
Again, one could have used  \cite[Theorems 1.1(1) and  1.2]{BP14},
 to compute the invariants 
$W_{X_{\R},\RP^2,0}(2c_1(X);s)$ with $s<2$ out of the
 values
$W_{\CP^2,\RP^2,0}(2[D];s)$,
 $W_{\CP^2,\RP^2,0}(4[D];s+3)$, and the equation
$$0=W_{X_{\R},\RP^2,0}(2c_1(X);0)
-2W_{\CP^2,\RP^2,0}(4[D];s+3) +2 W_{\CP^2,\RP^2,0}(2[D];s) $$
given by the combination of $(\ref{equ:sign W})$, Theorem
\ref{thm:main1}, and \cite[Theorems 1.1(1) and 1.2]{BP14}.
\end{rem}

\subsection{$\R$-minimal conic bundles}
Given $n\in\Z_{>0}$, let  $X_{n}$  be the  conic bundle
 given in an affine
chart by the real equation
$$y^2+z^2=-\prod_{i=1}^{2n}(x- a_i)$$
where $a_1<a_2\ldots<a_{2n}$ are distinct real numbers. The conic
bundle structure is given by the map  $\rho:X_{n}\to \C P^1$ that
forgets the $(y,z)$ coordinates.
Restricting 
 the standard complex conjugation on $\C^3$
 to $X_n$ turns this latter into a
 real algebraic surface $X_{n,\R}$, and
turns $\rho$ into a real map. Note that $X_{n,\R}$ is $\R$-minimal if and
only if $n\ge 2$.
The real part of $X_{n,\R}$ is composed of the $n$ spheres
$S_{i}=\rho^{-1}([a_{2i-1},a_{2i}])\cap \R^3$.
Next lemma is a straightforward combination of
 \cite[Propositions 4.2, 4.3, and 4.5, and Corollary 2.6]{BP14}.
\begin{lemma}
Given $F\in H^{\tau_X}_2(X_{n,\R}\setminus S_1;\Z/2\Z)$, there
exists $R\subset \R X_{n,\R}\setminus S_1$ a set of connected components 
of $\R X_{n,\R}\setminus S_1$  such that
$$ |W_{X_{n,\R},s_1,F}(d;s)|=|W_{X_{n,\R},S_1,[R]}(d;s)| \qquad
\forall d\in H_2(X_{n,\R};\Z) \mbox{ and } \forall s\in \Z_{\ge 0}.$$
\end{lemma}
That is to say, the computation of all
 Welschinger invariants of  $X_{n,\R}$ can be reduced to the
 case when $F$ is realized by
the union of some
 connected components of $\R X_{n,\R}\setminus S_1$. 

Now we describe how to  construct recursively the real
varieties $X_{n,\R}$, up to deformation, using real surgeries along
Lagrangian spheres.
Let $X_{0,\R}=(\CP^1\times\CP^1,\omega_{FS}\oplus \omega_{FS},\tau_{S^1,0})$ be the quadric hyperboloid in $\CP^3$,
and
let $\rho: X_{0,\R}\to \CP^1$ be the (real) projection to the first
factor. 
Consider the blow-up of $X_{0,\R}$ at
two complex conjugated points of a fiber $\Phi$ of $\rho$.
Then $X_{1,\R}$ is the  surgery of  $X_{0,\R}$ along the strict
transform of $\Phi$. 
Suppose now that we have constructed  $X_{n,\R}$, and let $\Phi$ be a real
fiber of $\rho$ with  $\R \Phi=\emptyset$. 
Then $X_{n+1,\R}$ is the  surgery of  $X_{n,\R}$ along the strict
transform of $\Phi$ under the blow-up of $X_{n,\R}$ at
two complex conjugated points of $\Phi$. 

This recursive description of $X_{n,\R}$ exhibit the underlying
complex surface as the blow-up of 
$X_{0,\R}$ at $2n$ points, and
shows 
that one can denote the corresponding exceptional divisors 
$E_1,\ldots, E_{2n}$ such that 
\begin{equation}\label{equ:class sphere Xn}
[S_i]=[\Phi]-[E_{2i-1}]-[E_{2i}] \in H_2(X_n;\Z) 
\end{equation}
for some orientation on $S_i$, where $\Phi$ is a fiber of $\rho$.
Denoting by $\widetilde{c_1(X_0)}$ the pull back of $c_1(X_0)$ by this
sequence of blow-ups, we have
$$c_1(X_n)=\widetilde{c_1(X_0)}-\sum_{i=1}^{2n}[E_i].$$
In Proposition \ref{prop:conic c1+bC} we compute the invariants
$ W_{X_{n,\R},S_1,[R]}(c_1(X_n)+b[\Phi];s)$ recursively starting from the
case $n=1$, which is treated in next proposition.

\begin{prop}\label{prop:conic c1+bC n=1}
Let $b\in\Z_{\ge -2}$. If $s\in\{0,\ldots, b+1\}$, then we have
$$ W_{X_{1,\R},S_1,0}(c_1(X_1)+b[\Phi];s)= 2^{2b+2-s}.$$
If $s=b+2$, then we have
$$W_{X_{1,\R},S_1,0}(c_1(X_1)+b[\Phi];b+2)=\left\{\begin{array}{cl}
2^{b+1} &\mbox{if } b \mbox{ is odd},
\\ 0 & \mbox{otherwise}.
\end{array}\right. $$

\end{prop}
\begin{proof}
The real manifold $X_{1,\R}$ is the quadric
ellipsoid blown up at a pair of complex conjugated points. Hence  one could 
use the methods described in
\cite{Bru14} to prove the proposition.
We provide an alternative proof that illustrates applications of
the method exposed
in \cite{BP14}. Here we use notations introduced in Sections \ref{sec:sphere} and
 \ref{sec:sympl sum}. We also define $r=2b+5-2s$.

Let $\widetilde{X_{0,\R}}$ be the blow-up of $X_{0,\R}$ at two
complex conjugated points on a fiber $\Phi$ of $\rho$, and let $E$ be the
strict transform of $\Phi$. As explained above, 
$X_{1,\R}$ is, up to deformation,
 obtained as the real symplectic sum $\pi:\mathcal Z\to
D$
of $X_{0,\R}$ and 
$(\CP^1\times\CP^1,\omega_{0},\tau_{S^1,2})$ along
$E$.
Now
 choose $\x: D\to \mathcal Z$  a  set of $r$ real 
sections and $s$ pairs of conjugated real sections
such that  
$$|\x(0)\cap \widetilde{X_{0}}|=2b+4\qquad\mbox{and}\qquad
|\x(0)\cap (\CP^1\times\CP^1)|=1.$$
Let $\overline f:\overline C\to  \widetilde{X_{0}}\cup \CP^1\times\CP^1$
be  an element of  $\CC^\C(c_1(X_1)+b[\Phi],0,\x(0),\emptyset,J_0)$.
According to \cite[Proposition 3.7]{BP14},
 we have
$$\overline f_*[C_1]=c_1(X_1)+ b[\Phi] -[E]= 
\widetilde{c_1(X_0)}+ (b-1)[\Phi]
\in H_2(X_1;\Z).$$
According to \cite[Proposition 3.7]{BP14}, we have the two following
possibilities:
\begin{enumerate}
\item[$(i)$] $C_1$ is real and irreducible, and $\overline f(C_1)$ is tangent
  to $E$ at a real point. In this case, it follows from
  \cite[Proposition 3.10]{BP14} that the contribution to 
$ W_{X_{1,\R},S_1,0}(c_1(X_1)+b[\Phi];s) $
of the two deformations of $\overline f$ cancel each other.

\item[$(ii)$] $C_1$ is real and has two irreducible components $C'$ and
  $C''$. 
In this case,
$\overline f$ deforms into a unique real curve in
$\CC^\C(c_1(X_1)+b[\Phi],\x(t),0,\emptyset,J_t)$,  and
  $\overline f(C_0)$ intersects $E$ transversely in two 
  points.
\end{enumerate}
Hence it remains to estimate the contribution 
to 
$ W_{X_{1,\R},S_1,0}(c_1(X_1)+b[\Phi];s) $ of the maps in case $(ii)$ above.
None of the curves
 $\overline f(C')$ nor $\overline f(C'')$ 
intersects any of the two exceptional divisor coming from the blow-ups
of $X_{0,\R}$, therefore there exists $b_1\in\{0,\ldots,b+1\}$ such that
$$\overline f(C')=l_1 +b_1l_2 \qquad
\mbox{and}\qquad
\overline f(C'')=l_1 +(b+1-b_1)l_2$$
(we still denote by $l_i$ the strict transform the class $l_i$ under
the blow-up map).
The maps $\overline f_{|C'}$ and $\overline f_{|C''}$ can be
constrained respectively by at most $2b_1+1$ and $2b-2b_1+3$ points.
Since $\overline f_{|C_1}$ is constrained by the $2b+4$ points in
$\x(0)\cap \widetilde{X_{0,\R}}$, we deduce that $\overline f_{|C'}$ and $\overline f_{|C''}$ are
constrained respectively by exactly $2b_1+1$ and $2b-2b_1+3$ points.

Suppose that $r\ge 3$. In this case we have $\x(0)\cap
\R\widetilde{X_{0,\R}}\ne\emptyset$, which in particular implies that
both curves $C'$ and $C''$ and
both maps $\overline f_{|C'}$ and $\overline f_{|C''}$ are real.
Let us  choose $p\in \x(0)\cap \R\widetilde{X_{0,\R}}$. 
Without loss of generality, we may assume
that $p\in \overline f(C')$. Since the Gromov-Witten invariant of 
$\CP^1\times\CP^1$ for the class 
$l_1+b_1l_2$ is equal to $1$ for any
$b_1\ge 0$, the map $\overline f$ is determined by the real subset of
$\x(0)\setminus \{p_0,p\}$ that is contained in $\overline
f(C')$. Let $\Im\x(0)$ be the set of pairs of conjugated points of 
$\x_0$.
Given any sets $A\subset \R\x(0)\setminus \{p_0,p\}$ and 
 $B\subset \Im\x(0)$, the pair among the two pairs $(A,B)$ and 
$(\R\x(0)\setminus(A\cup \{p_0,p\}),\Im\x(0)\setminus B)$ for which
the first component has an even cardinal corresponds to a real subset of
$\x(0)\setminus \{p_0,p\}$ that is contained in $\overline
f(C')$ for some map $\overline f$ as above. Hence there exist
exactly
$$2^{r+s-3} $$
such maps.
By the adjunction formula, any 
irreducible $J$-holomorphic curve in $\CP^1\times\CP^1$ realizing the class 
$l_1+b_1l_2$ is smooth. Hence the curve $\overline f(C_1)$ has no
elliptic  real node, and the proposition is proved in the case when
$r\ge 3$.

Suppose now that  $r=1$. In this case  $\R\x(0)\cap
\widetilde{X_{0}}=\emptyset$, and 
there is no possibilities to
decompose $x(0)\cap \widetilde{X_{0}}$ into the disjoint union of
two real subsets of odd cardinality.  As a consequence the two
curves $C'$ and $C''$  and the two 
maps $\overline f_{|C'}$ and $\overline f_{|C''}$ are complex conjugated, 
and each map is
determined by the point in each pair of $\x(0)\setminus p_0$ through
which it passes.
Hence there are 
exactly 
$$2^{s-1}=2^{b+1}$$
such maps $\overline f$ if $b+2$ is odd, and $0$ such maps otherwise. 
Since $\R X_{0,\R}$ is null-homologous in
$H_2(X_0;\Z/2\Z)$, the curve $\overline f(C')$ has an even number of
intersection points with $\R \widetilde{X_{0,\R}}$. Hence the proposition is
proved as well when
$r=1$.
\end{proof}

\begin{prop}\label{prop:conic c1+bC}
Let $n\ge 2$, $b\ge n-3$ and $s\le b-n+3$ be three integer
numbers, and let $R$
be the union 
of some connected components of $\R X_{n}\setminus S_1$. Then we have
$$ W_{X_{n,\R},S_1,[R]}(c_1(X_n)+b[\Phi];s)= \left\{\begin{array}{ll}
2^{2b+2-s} &\mbox{if }R=\emptyset;
\\(-1)^{b+1} \ 2^{b+n-1} &\mbox{if } R=\R X_{n}\setminus S_1 \mbox{ and } s=b-n+3;
\\ 0 & \mbox{otherwise}.
\end{array}
\right.$$
\end{prop}
\begin{proof}
The fact that $ W_{X_{n,\R},S_1,[R]}(c_1(X_n)+b[\Phi];s)=0$ if
$R\ne\emptyset$ and $s<b+n-3$ follows from \cite[Theorems 1.1 and 1.2]{BP14}.
Given $n\ge 2$,  
Theorem \ref{thm:main1} 
implies that
\begin{equation}\label{equ:rec Xn}
 W_{X_{n,\R},S_1,F}(c_1(X_n)+b[\Phi];s)=
W_{X_{n-1,\R},S_1,F}(c_1(X_{n-1})+b[\Phi];s+1) 
\end{equation}
$$ \qquad  \qquad  \qquad  \qquad  \qquad  \qquad  \qquad  \qquad  \qquad  \qquad +
2W_{X_{n-1,\R},S_1,F}(c_1(X_{n-1})+(b-1)[\Phi];s) .
$$
The proposition in the case when $R=\emptyset$ follows now by
induction on $n$ from 
$(\ref{equ:rec Xn})$ and Proposition \ref{prop:conic c1+bC n=1}.

Let us assume that $R=\R X_{n}\setminus S_1$ and $s=b-n+3$. 
According
to  Remark \ref{rem:two special F},
identity $(\ref{equ:rec Xn})$ becomes
\begin{equation}\label{equ:rec Xn2}
 W_{X_{n,\R},S_1,[\R X_{n,\R}\setminus S_1]}(c_1(X_n)+b[\Phi];s)=
W_{X_{n-1,\R},S_1,[\R X_{n-1,\R}\setminus S_1]}(c_1(X_{n-1})+b[\Phi];s+1) 
\end{equation}
$$ \qquad  \qquad  \qquad  \qquad  \qquad \qquad   \qquad  \qquad  \qquad  \qquad  \qquad -
2W_{X_{n-1,\R},S_1,[\R X_{n-1,\R}\setminus S_1]}(c_1(X_{n-1})+(b-1)[\Phi];s) .
$$
If
$n=2$, then the proposition  follows from 
$(\ref{equ:rec Xn2})$ and  Proposition \ref{prop:conic c1+bC n=1}.
If $n\ge 3$, then the proposition  follows by induction on $n$ from 
$(\ref{equ:rec Xn2})$.

The proof of the proposition when  
$\emptyset\ne R\subsetneq(\R X_{n}\setminus S_1)$ and $s=b-n+3$ is
analogous: 
the case $n=3$ follows from $(\ref{equ:rec Xn})$ and the above computation
of $W_{X_{2,\R},S_1,[\R X_{2,\R}\setminus S_1]}(c_1(X_2)+b[\Phi];s)$, and the
case $n\ge 4$ follows by induction on $n$ from 
$(\ref{equ:rec Xn})$.
\end{proof}

\subsection{Del~Pezzo surfaces of degree 1}

Recall that there exist 11 deformation classes of  real
del~Pezzo surfaces of degree 1. 
In \cite[Example 7.9]{Bru14}, the values of
 $W_{X_{\R},L,[\R X\setminus  L]}(2c_1(X);0)$ are computed for 6
of these deformation classes.
Here we treat   the 5 remaining cases.

Recall also that  the  real  conic bundle $X_{n,\R}$ with 
$n\le 3$
is a del~Pezzo surface of degree $8-2n$, and that $\Phi$
denotes a generic fiber of $X_{n,\R}$. We denote by
$\RP^2_k$ the real blow-up of $\RP^2$ at $k$ points (i.e. we replace
$k$ disjoint disks in $\RP^2$ by $k$ Mobius strips). 
By ``$Y_\R$ is obtained by a  surgery of $X_\R$ along the class
$\gamma$'', we mean that $Y_\R$ is obtained by a  surgery of a
deformation of $X_\R$ for which the class $\gamma$ is realized by a
real Lagrangian sphere. In what follows, the existence of such
deformation is guaranteed by the classification up to deformations of
real algebraic rational surfaces.
We define the
following real del~Pezzo surfaces of degree 1:
\begin{itemize}
\item $Y_{1,\R}$ is the real blow-up of $X_{2,\R}$ at one real point
  on a connected component of $\R X_{2,\R}$ and at two real points
  on the other connected component. Denoting respectively by
  $\widetilde E_1,\widetilde E_2,\widetilde E_3$ the
  three corresponding exceptional divisors, we have
$$H_2^{-\tau_{Y_1}}(Y_{1,\R}; \Z)=\Z c_1(Y_1)\oplus \Z [\Phi] \oplus \Z
  [\widetilde E_1]\oplus \Z [\widetilde E_2]\oplus \Z [\widetilde E_3]\qquad
\mbox{and}\qquad \R Y_{1,\R}=\RP^2\sqcup \RP^2_1.$$

\item $Y'_{1,\R}$ is the real blow-up of $X_{2,\R}$ at three real points
  on a connected component of $\R X_{2,\R}$.
Denoting  by $\widetilde E_1,\widetilde E_2,\widetilde E_3$ the
  three corresponding exceptional divisors, we have
$$H_2^{-\tau_{Y'_1}}(Y'_{1,\R}; \Z)=\Z c_1(Y'_1)\oplus \Z [\Phi] \oplus \Z [\widetilde E_1]\oplus
  \Z [\widetilde E_2]\oplus \Z [\widetilde E_3]\qquad
\mbox{and}\qquad \R Y'_{1,\R}=S^2\sqcup \RP^2_2.$$

\item $Y''_{1,\R}$ is the real blow-up of $X_{1,\R}$ at a real point and two pairs of
  complex conjugated points. Denoting respectively by $\widetilde E_1,\ldots,\widetilde E_5$ the
  five corresponding exceptional divisors, we have
$$H_2^{-\tau_{Y''_1}}(Y''_{1,\R}; \Z)=\Z c_1(Y''_1)\oplus \Z [\Phi] \oplus \Z
  [\widetilde E_1]\oplus \Z([\widetilde E_2]+ [\widetilde E_3])\oplus
  \Z([\widetilde E_4]+ [\widetilde E_5])$$
and
$$\R Y''_{1,\R}=\RP^2.$$

\item $Y_{2,\R}$ is the real blow-up of $X_{2,\R}$ at a real point
  on a connected component of $\R X_{2,\R}$, and a pair of
  complex conjugated points. Denoting respectively by $\widetilde E_1,\widetilde E_6,\widetilde E_7$ the
  three corresponding exceptional divisors, we have
$$H_2^{-\tau_{Y_2}}(Y_{2,\R}; \Z)=\Z c_1(Y_2)\oplus \Z [\Phi] \oplus \Z [\widetilde E_1]\oplus \Z([\widetilde E_6]+ [\widetilde E_7])\qquad
\mbox{and}\qquad \R Y_{2,\R}=\RP^2\sqcup S^2.$$

\item $Y_{3,\R}$ is the  surgery of  $Y_{2,\R}$  along the class 
$[\Phi]-[\widetilde E_6]-[\widetilde E_7]$.
We have
$$H_2^{-\tau_{Y_3}}(Y_{3,\R}; \Z)=\Z c_1(Y_2)\oplus \Z [\Phi] \oplus \Z [\widetilde E_1]\qquad
\mbox{and}\qquad \R Y_{3,\R}=\RP^2\sqcup 2S^2.$$
Note that $Y_{3,\R}$ is the blow-up of $X_{3,\R}$ at a real point.

\item $Y_{4,\R}$ is the surgery of  $Y_{3,\R}$  along the class 
$c_1(Y_3)-[\Phi]+[\widetilde E_1]$.
We have
$$H_2^{-\tau_{Y_4}}(Y_{4,\R}; \Z)=\Z c_1(Y_4) \oplus \Z [\widetilde E_1]\qquad
\mbox{and}\qquad \R Y_{4,\R}=\RP^2\sqcup 3S^2.$$
Note that $Y_{4,\R}$ is the blow-up  at a real point of an $\R$-minimal real
del~Pezzo surface of degree $2$.

\item $Y_{5,\R}$ is the  surgery of  $Y_{4,\R}$  along the class 
$c_1(Y_4)-[\widetilde E_1]$.
We have
$$H_2^{-\tau_{Y_5}}(Y_{5,\R}; \Z)=\Z c_1(Y_5) \qquad
\mbox{and}\qquad \R Y_{5,\R}=\RP^2\sqcup 4S^2.$$
Note that $Y_{5,\R}$ is the  only (up to deformation) $\R$-minimal real
del~Pezzo surface of degree $1$.
\end{itemize}
Note that $Y''_{1,\R}$ is also the real blow-up of $\RP^2$ at four
pairs of complex conjugated points, and that $Y_{2,\R}$ can also be
constructed as the  surgery of $Y_{1,\R}$ or  $Y'_{1,\R}$
 along the class 
$[\Phi]-[\widetilde E_2]-[\widetilde E_3]$, as well as of $Y''_{1,\R}$ along the class 
$[\Phi]-[\widetilde E_4]-[\widetilde E_5]$. 
Hence we have the following diagram
$$\begin{array}{lllllllll}
\\ & & & & & & Y'_{1,\R} & & 
\\
\\  & & & & & &\left\uparrow\rule{0pt}{0.8cm}\right.
\scriptstyle [\Phi]-[\widetilde E_2]-[\widetilde E_3] & & 
\\
\\ Y_{5,\R} & \xrightarrow[]{c_1(Y_4)-[\widetilde E_1]} &
Y_{4,\R} & \xrightarrow[]{c_1(Y_3)-[\Phi]+[\widetilde E_1]} &
Y_{3,\R} & \xrightarrow[]{[\Phi]-[\widetilde E_6]-[\widetilde E_7]} &
Y_{2,\R} & \xrightarrow[]{[\Phi]-[\widetilde E_2]-[\widetilde E_3]} &
Y_{1,\R} 
\\
\\  & & & & & &\left\downarrow \rule{0pt}{0.8cm}\right.
\scriptstyle{[\Phi]-[\widetilde E_4]-[\widetilde E_5]}  & & 
\\\\ & & & & & & Y''_{1,\R} & & 
\end{array} $$
where $Y_\R \xrightarrow[]{\gamma}  X_\R$ means that $X_\R$ and $Y_\R$ are related
by a surgery along  the class $\gamma$, and that
$\chi(\R Y)=\chi(\R X)+2$.

\begin{lemma}\label{lem:euler char}
Let $Y_\R$ be a real del~Pezzo surface of degree 1, $L$ be a connected
component of $\R Y_\R$, and $\gamma\in
H_2^{-\tau_Y}(Y;\Z)$ be a class realized by an exceptional rational curve.
Then
$$W_{Y_\R,L, [\R Y\setminus L]}(c_1(Y)+\gamma;0)=-\chi(\R Y)+1. $$
\end{lemma}
\begin{proof}
Let us denote by $Z_\R$ be the real del~Pezzo surface of degree 2
obtained by blowing down the exceptional rational curve realizing the
class $\gamma$. Then the class $c_1(Y)+\gamma$ is the pull-back of the
class $c_1(Z)$, and we have
$$W_{Y_\R,L, [\R Y\setminus L]}(c_1(Y)+\gamma;0)=W_{Z_\R,L, [\R
    Z\setminus L]}(c_1(Z);0)=-\chi(\R Z)+2. $$
This latter equality can be proved for example by 
an Euler
characteristic computation as in \cite[Proposition 4.7.3]{DK}.
Now the result follows since $\chi(\R Y)=\chi(\R Z)-1$.
\end{proof}

\begin{prop}
Let $Y_\R$ be a real del~Pezzo surface of degree 1, and $L$ be a connected
component of $\R Y$. Then we have the following Welschinger invariants 
\begin{center}
\begin{tabular}{ |c|c| c|c|c|c|c|c|}
\hline $ Y_\R $  & $Y_{5,\R}$& $Y_{4,\R}$& $Y_{3,\R}$& $Y_{2,\R}$& $Y_{1,\R}$& $Y'_{1,\R}$ &$Y''_{1,\R}$
\\\hline
$W_{Y_\R,L, [\R Y\setminus L]}(2c_1(Y);0)$ &  30 & 18  & 10 &6
& 6 & 6&6 
\\\hline
\end{tabular}
\end{center}
In particular, $W_{Y_\R,L, [\R Y\setminus L]}(2c_1(Y);0)$ does not
depend on the choice of $L$.
\end{prop}
\begin{proof}
The cases  $Y_{1,\R}$ and $Y'_{1,\R}$ have been computed in
\cite{Bru14}\footnote{Note that in the published version of \cite{Bru14},
the number $W_{\widetilde X_{8,1}(4),\overline{L}_1}(4[D]-\sum_{i=1}^8[\widetilde
  E_i]-2[\widetilde E_9])$ is erroneously claimed to be equal to $4$
instead of $6$.}. 
Let $X_\R$ and
$Y_\R$ be two real algebraic surfaces as in the proposition and such
that
$Y_\R \xrightarrow[]{\gamma}  X_\R$.
Suppose that there exists an immersed
 $J$-holomorphic curve $C$ in
$X$ realizing the class $2c_1(X)-k\gamma$.
Since the class $\gamma$ 
satisfies
$$\gamma^2=-2\qquad\mbox{and}\qquad \gamma\cdot c_1(X)=0,$$
it follows from the adjunction formula that $C$ has genus at most
$2-2k^2$. In particular we deduce  that $k=0$ or $1$.
Hence by Theorem \ref{thm:main1} we have
$$W_{X_\R,L, [\R X\setminus L]}(2c_1(X);0)=W_{Y_\R,L, [\R
    Y\setminus L]}(2c_1(Y);0)
-2W_{Y_\R,L, [\R Y\setminus L]}(2c_1(Y)-\gamma;0). $$
Now the result follows from the values of
$W_{Y_{1,\R},L, [\R Y_{1,\R}\setminus L]}(2c_1(Y_1);0)$ and
$W_{Y'_{1,\R},L, [\R Y'_{1,\R}\setminus L]}(2c_1(Y'_1);0)$, and from Lemma \ref{lem:euler char}.
\end{proof}

\bibliographystyle {alpha}
\bibliography {../../Biblio.bib}

\end{document}